\documentclass[a4paper]{article}
\usepackage{graphicx}
\usepackage{amssymb}
\usepackage{amsmath}
\usepackage{amsthm}
\usepackage[all,2cell]{xy}
\usepackage{enumerate}
\usepackage{tikz}
\usepackage{hyperref}
\UseAllTwocells
\input xy

\theoremstyle{plain}
\newtheorem{thm}{Theorem}[section]
\newtheorem{cor}[thm]{Corollary}
\newtheorem{lem}[thm]{Lemma}
\newtheorem{prop}[thm]{Proposition}

\newtheorem*{conjecture}{Conjecture}

\theoremstyle{definition}
\newtheorem{defn}[thm]{Definition}
\newtheorem{rem}[thm]{Remark}
\newtheorem{example}[thm]{Example}

\newenvironment{tz}[1][]{\begin{aligned}\begin{tikzpicture}[#1]}{\end{tikzpicture}\end{aligned}}

\newcommand\ba{\begin{aligned}}
\newcommand\ea{\end{aligned}}

\newcommand{\ig}[1]{\begin{aligned}\includegraphics{#1}\end{aligned}}

\newcommand\XXX[1]{}

\newcommand\id{\operatorname{\text{\rm id}}}

\newcommand\be{\begin{equation}}

\newcommand\ee{\end{equation}}

\newcommand{\Tr}{\operatorname{\text{\rm Tr}}}

\newcommand{\Rep}{\operatorname{\text{\rm Rep}}}

\newcommand{\Hom}{\operatorname{\text{\rm Hom}}}

\newcommand{\Aut}{\operatorname{\text{\rm Aut}}}

\newcommand{\End}{\operatorname{\text{\rm End}}}

\newcommand{\comments}[1]{}

\def\clap#1{\hbox to 0pt{\hss#1\hss}}

\newcommand{\lag}[2]{ \xy {\ar@{->}^(0.45){\scriptscriptstyle #2}_(0.4){\scriptscriptstyle
#1} (6,0)*{}; (0,0)*{}};
\endxy}

\newcommand{\biglag}[2]{ \xy {\ar@{->}^(0.45){\scriptscriptstyle #2}_(0.4){\scriptscriptstyle
#1} (15,0)*{}; (0,0)*{}};
\endxy}

\newcommand{\nlag}[3]{{ \begin{xy}   (4,-0.7)*{\scriptscriptstyle #1}="d"; (0,-0.7)*{\scriptscriptstyle #2}="a"; (1.2,-0.7)*{\scriptscriptstyle <}  {\ar@{-}_{\scriptscriptstyle \; \; #3} "d"; "a"} \end{xy}}}

\newcommand{\cat}{\mathbf}

 \reversemarginpar
\SelectTips{cm}{}

\newcommand{\Addresses}{{
  \bigskip
  \footnotesize

  \textsc{Mathematical Institute, University of Oxford and Mathematics Division, \-Stellenbosch University}\par\nopagebreak
  \textit{E-mail address}, : \texttt{brucehbartlett@gmail.com}
}}

\begin{document}


\title{Fusion categories via string diagrams}

%
\author{Bruce Bartlett} 

\date{}

\maketitle

\begin{abstract}    
We use the string diagram calculus to give graphical proofs of the basic results of Etingof, Nikshych and Ostrik \cite{eno02-ofc} on fusion categories. These results include: the quadruple dual is canonically isomorphic to the identity, positivity of the paired dimensions, and Ocneanu rigidity. We introduce the pairing convention as a convenient graphical framework for working with fusion categories. We use this framework to express the pivotal operators as a product of the apex associator monodromy and the pivotal indicators. We also characterize pivotal structures as solutions of an explicit set of algebraic equations over the complex numbers. 
\end{abstract}

\section{Introduction}
In this paper we fix our ground field to be $\mathbb{C}$. A fusion category is a rigid semisimple linear monoidal category with finitely many isomorphism classes of simple objects and whose unit object is simple. A basic example of a fusion category is the category of representations of a finite group; in this way the study of fusion categories can be regarded as a common generalization of group theory and representation theory. For an overview on fusion categories, see \cite{muger2008tensor, calaque2004lectures}.

Many of the basic results on fusion categories were obtained by Etingof, Nikshych and Ostrik in \cite{eno02-ofc}. These results include the fact that the quadruple dual functor is canonically isomorphic to the the identity functor, positivity of the paired dimensions, and Ocneanu rigidity.

These results were proved in \cite{eno02-ofc} using the theory of weak Hopf algebras. The idea is that since every fusion category can be expressed as the category of finite-dimensional representations of some semisimple weak Hopf algebra \cite{eno02-ofc}, one may prove results about fusion categories by translating them into the language of weak Hopf algebras. One motivation for this paper, building on \cite{hh09-snbfcor3, k05-aes, muger2003-subfactors, bv13-ocfc}, was to give a unified account of how these results can be proved directly in the fusion category itself, using the graphical calculus of string diagrams \cite{js91-gtc, js-gtcii, s11-sgc}. 

Another motivation has been the string-net description for the vector spaces in the Turaev-Viro model of 3-dimensional topological quantum field theory \cite{kirillov2011-string-net}, where the graphical calculus plays a prominent role. 

However, the main motivation has been to investigate the conjecture of Etingof, Nikshych and Ostrik, that every fusion category admits a pivotal structure \cite{eno02-ofc}, as we now explain.

\subsection*{The pairing convention}
In the string diagram calculus, the objects $V$ of the fusion category $\cat{C}$ are used as labels for strands in the plane. The main feature is rigidity --- the notion that the strands are oriented, and that there are left and right dual structure maps, drawn as
\[
 \left( \ig{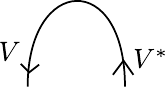}\, , \, \ig{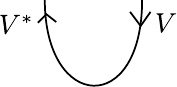} \right) \quad \text{and} \quad \left( \ig{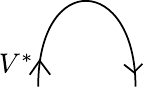} \, , \, \ig{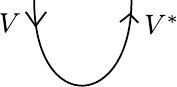} \right)
\]
respectively, satisfying `snake equations' such as
\[
\ba \ig{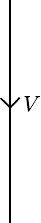} \ea = \ba \ig{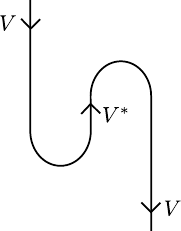} \ea.
\]
One can also form closed loops, which involves pairing left and right dual structure maps together:
\begin{equation} \label{loop_a}
\ba \ig{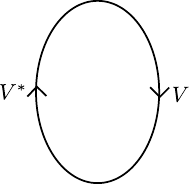} \ea \in \Hom(1,1) \,.
\end{equation}
In order to make this well-defined, it is usually imagined that one needs a {\em pivotal structure} on the fusion category, that is, a monoidal natural isomorphism $\gamma : \id \Rightarrow **$ from the identity functor on the fusion category to the double dual functor (see Section \ref{Explicit_equations}). Etingof, Nikshych and Ostrik have conjectured that such a structure always exists:

\begin{conjecture}[\cite{eno02-ofc}] Every fusion category admits a pivotal structure.
\end{conjecture}

The main idea we use in this article is that, in fact, there is a well-defined way to make sense of diagrams which employ both left and right dual structure maps together, such as \eqref{loop_a}, {\em even without} a pivotal structure. The rule (at least for simple objects $V$) is that if left and right dual structure maps appear in a diagram, then they cannot be chosen independently but must be correlated in the following way: when paired together, they must give the {\em fusion dimension} of $V$. The fusion dimension of $V$ is a certain positive number canonically associated to $V$ in a fusion category which we define in Section \ref{fusion_frobenius}; it does not make use of a pivotal structure. We call this the pairing convention, and it allows us to unambiguously evaluate diagrams such as
\[
\ig{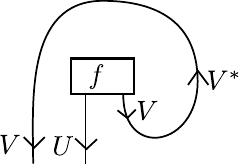}
\]
in a `bare' fusion category, which would otherwise not be well-defined. 

\subsection*{Pivotal operators}

Using the pairing convention, we can give a simple diagrammatic definition of the {\em pivotal operators} of a fusion category, which are certain canonically defined linear maps
\[
 T^A_{BC} : \Hom(A, B \otimes C) \rightarrow \Hom(A, B \otimes C)
\]
associated to every triple of objects $A$,$B$ and $C$ in a fusion category. One advantage of our framework is that these maps are canonical, not depending on any choices of duals for their definition. The pivotal operators play the role of the double dual functor, and the string diagram calculus then gives a simple proof that $T^A_{BC}$ squares to the identity, see Theorem \ref{piv_involution_thm}. The idea of this graphical proof is originally due to Hagge and Hong \cite{hh09-snbfcor3}.

\subsection*{Main results}
Our main results are as follows. Firstly, we show that every fusion category $\cat{C}$ over $\mathbb{C}$ comes equipped with a canonical monoidal endofunctor $\mathcal{T}$, whose underlying functor is the identity, whose action on the hom-sets $\Hom(X_i, X_j \otimes X_k)$ between tensor products of simple objects is given graphically by the pivotal operators, and which squares to the identity (Theorem \ref{piv_operator_C}). 

Secondly, we use this viewpoint to give graphical proofs of the following results from \cite{eno02-ofc}: 
\begin{itemize}
\item the quadruple dual is monoidally naturally isomorphic to the identity (Corollary \eqref{explicit_eno_cor})
\item positivity of the paired dimensions (Theorem \ref{realpositive}), 
\item a pseudo-unitary fusion category is spherical (Corollary \ref{pseudo_cor})
\item the sphericalization of a fusion category is spherical (Corollary \ref{spher_cor})
\item Ocneanu rigidity, i.e. the vanishing of the Davydov-Yetter cohomology (Theorem \ref{ocneanu})
\end{itemize}
Thirdly, in Theorem \ref{formula_piv} we explicitly compute the pivotal operators in terms of the associators of the fusion category as a product of the {\em pivotal indicators}, which we introduce in Section \ref{piv_ind_sec}, with the {\em apex associator monodromy}, which we introduce in Section \ref{Form_piv_op}. 

Finally, we characterize pivotal structures as solutions of an explicit set of equations over the complex numbers (Theorem \ref{piv_structure_thm}), refining a formula of Wang \cite[Prop 4.17]{wang2010-topological}.

\subsection*{Outline of paper}

In Section \ref{string_sec} we give our conventions on string diagrams. In Section \ref{stss} we define the pairing convention and the pivotal operators, show that the paired dimensions are positive, and show how to interpret the pivotal operators as a monoidal functor. In Section \ref{explicit_sec} we compute the pivotal operators explicitly in terms of the associators. In Section \ref{piv_struc_sec} we express pivotal structures as solutions of certain explicit algebraic equations, and formulate the sphericalization of a fusion category from our approach. Finally in Section \ref{ocneanu_sec} we give a graphical proof of Ocneanu rigidity.

\section{String diagrams} \label{string_sec}
In this section we explain our conventions regarding the string diagram calculus.

\subsection{Conventions for general monoidal categories} \label{string_conventions}
The string diagram calculus for a monoidal category $\cat{C}$ is well-known (see \cite{s11-sgc} or \cite[Chap. 4]{BBthesis} for an overview). Each diagram refers to a certain morphism in $\cat{C}$. Our diagrams go from top to bottom, so that a morphism $f : A \rightarrow B \otimes C$ is drawn as
\[
\ig{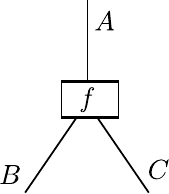} \, .
\] 
In these pictures there is no explicit parenthesis scheme on the input and output tensor products. Thus it is often supposed that one needs a a strict monoidal category in order to make this calculus well-defined. In fact, this is not so --- the calculus makes perfect sense when reasoning about {\em equations} between morphisms, which is all we will ever need. For instance, suppose $f : A \rightarrow A' \otimes E$, $g : B \otimes C \rightarrow D \otimes C'$, $h: E \otimes D \rightarrow B'$, $k : A \otimes (B \otimes C) \rightarrow F \otimes C'$, $l: F \rightarrow A' \otimes B'$, and consider the following equation:
\[
\ig{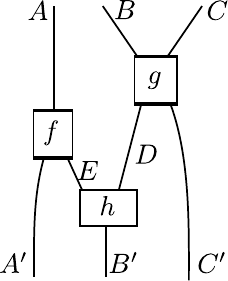} = \ig{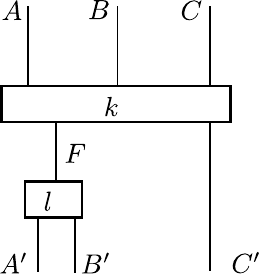}
\]
How are we to interpret this? This equation should be regarded as an infinite number of equations, one for each fixed parenthesis scheme (this may include insertions of the unit object $1$ of the monoidal category) of the source and target objects, which are taken to be the same on both sides of the equation. For instance, $((1 \otimes A) \otimes (B \otimes 1)) \otimes C$ is a possible choice of parenthesis scheme for the source object in the above equation. For each such choice of parenthesis schemes, one should insert an appropriate sequence of associators and unit isomorphisms in order to make each side well-typed. Coherence for monoidal categories \cite{ml97-cwm} guarantees that the evaluation of the resultant morphism is independent of this choice. Each equation implies all the others (see \cite[Chap. 4]{BBthesis}). 

A monoidal category $\cat{C}$ is called {\em right rigid} if for each object $V$ there exists an object $V^*$ which can be equipped with unit and counit maps $\eta : 1 \rightarrow V^* \otimes V$, $\epsilon : V \otimes V^* \rightarrow 1$ such that the rigidity equations hold:
\begin{equation} \label{snake_eqns}
\ig{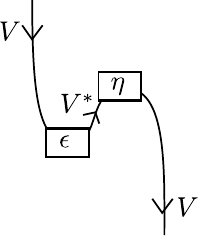} = \ig{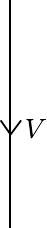} \quad , \quad \ig{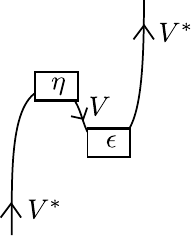} = \ig{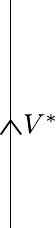}
\end{equation}
Note that we have added orientations to the strands in \eqref{snake_eqns}, this is simply a visual aid. The monoidal category $\cat{C}$ is called {\em left rigid} if for each object $V$ there exists an object ${}^* V$ which can be equipped with unit and counit maps $n : 1 \rightarrow V \otimes {}^* V$, $e : {}^* V \otimes V \rightarrow 1$ satisfying the horizontally-reversed versions of \eqref{snake_eqns}. We say that $\cat{C}$ is {\em rigid} if it is both right rigid and left rigid. Our approach is to avoid making fixed initial choices of duals for each object, to keep constructions as canonical as possible (thus we do not refer to a right dual functor $* : \cat{C} \rightarrow \cat{C}^\text{op}$). 

Formally, the invariance of the evaluation of string diagrams under diffeomorphisms of the plane can be expressed as follows. We imagine each diagram as living in an outer disk $D$, with strands labelled by objects of the category, and coupons $C_i$ (drawn as discs or square regions) labelled by morphisms, which go from the tensor product of the objects labelling the strands in the northern hemisphere of the coupon to the tensor product of the objects labelling the strands in the southern hemisphere of the coupon. We define a {\em rigid isotopy} of such a diagram to be a smooth path of diffeomorphisms $\phi_t : D \rightarrow D$, $t \in [0,1]$ such that $\phi_0$ is the identity, $\phi_t$ is the identity on the boundary of $D$ for all $t$ and such that $\phi_t$ restricted to each coupon $C_i$ is a translation for all $t$.

The following theorem is well-known and combines coherence for monoidal categories, rigidity, and the interchange law.
\begin{thm}[{see eg. \cite{s11-sgc}}] \label{invariance_rigid}
The evaluation of a string diagram taking labels in a rigid monoidal category is invariant under diffeomorphisms which are rigid isotopic to the identity. 
\end{thm}
\begin{rem}
In a companion paper, motivated by \cite{kirillov2011-string-net}, we will express the diffeomorphism invariance for string diagrams taking labels in a {\em fusion category} in the language of spin structures. \label{spin_remark}
\end{rem}

\subsection{Conventions for fusion categories}
A fusion category is a rigid semisimple linear monoidal category with finitely many isomorphism classes of simple objects and whose unit object is simple \cite{e06-2vect}. In a fusion category, the associators feature explicitly in the graphical calculus as follows. Suppose we fix a representative set of simple objects $X_i$, $i \in I$, and choose trivalent bases
\begin{equation} \label{basis_choice}
 e_\alpha : X_i \rightarrow X_j \otimes X_k  \quad \text{ drawn as } \ig{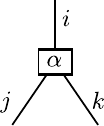}
 \end{equation}
for each hom-set $\Hom(X_i, X_j \otimes X_k)$. (We will always choose the canonical basis elements for the 1-dimensional hom-spaces $\Hom(1, 1 \otimes 1)$, $\Hom(X_i, 1 \otimes X_i)$ and $\Hom(X_i, X_i \otimes 1)$). Then we can form two different bases for $\Hom(X_i, X_j \otimes (X_k \otimes X_l))$ and so we have a change of basis transformation
\begin{equation} \label{associator_eqn}
\ig{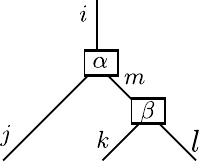} = \sum_{\gamma, n, \delta} (F^i_{jkl})^{\gamma n \delta}_{\alpha m \beta} \ig{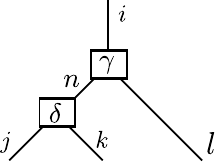} \, .
\end{equation}
Note that according to our diagram conventions, in the parenthesis scheme $X_j \otimes (X_k \otimes X_l)$ for the output objects, the left-hand side diagram of \eqref{associator_eqn} evaluates as
\[
X_i \stackrel{e_\alpha}{\longrightarrow} X_j \otimes X_m \stackrel{\id \otimes e_\beta}{\longrightarrow} X_j \otimes (X_k \otimes X_l )
\]
while the right-hand side diagram evaluates as
\[
X_i \stackrel{e_\gamma}{\longrightarrow} X_n \otimes X_l \stackrel{e_\delta \otimes \id}{\longrightarrow} (X_j \otimes X_k) \otimes X_l \stackrel{a_{X_j, X_k, X_l}}{\longrightarrow} X_j \otimes (X_k \otimes X_l) 
\]
where $a$ is the associator of the fusion category. We call the scalars $(F^i_{jkl})^{\gamma n \delta}_{\alpha m \beta}$ the {\em associator matrix elements}. 

Given a trivalent basis choice \eqref{basis_choice}, we define the {\em dual basis} $\hat{e}_\alpha : X_j \otimes X_k \rightarrow X_i$ as the one satisfying $\hat{e}_\beta \circ e_\alpha = \delta_{\alpha \beta} \id_{X_i}$. Graphically, this is drawn as:
\begin{equation} \label{dual_basis}
\left\{ \ig{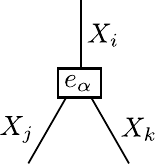} \right\} \text{ and } \left\{ \ig{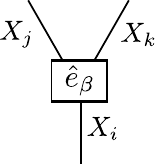} \right\} \text{ satisfying } \ig{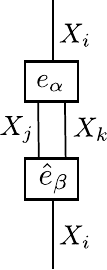} = \delta_{\alpha \beta} \ig{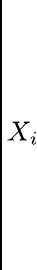} \,.
\end{equation}
We will often write both $e^\alpha$ and $\hat{e}_\beta$ simply as $\alpha$ and $\hat{\beta}$ respectively; also we often write $i$ instead of $X_i$. Basis vectors satisfying \eqref{dual_basis} provide a resolution of the identity, 
\begin{equation} \label{res_id}
\sum_{i, \alpha} \ig{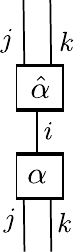} = \ig{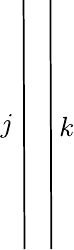} \, .
\end{equation}

\begin{example} The Yang-Lee category (see eg. \cite{wang2010-topological}) has two simple objects $1$ (drawn as a dotted line) and $\tau$ (drawn as a solid line) with $\tau^2 = 1 + \tau$. So, the only nontrivial trivalent basis elements are
\[
 \ig{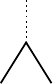} \quad \text{and} \quad \ig{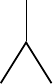}.
\]
The associators are
\begin{align}
\ig{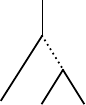} & = a \ig{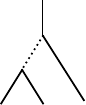} + \ig{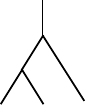} \label{yl1}  \\
\ig{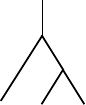} & = a \ig{d86.pdf} - a \ig{d87.pdf}
\end{align}
and
\begin{equation} \label{yl2}
 \ig{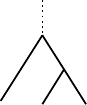} = \ig{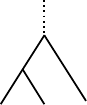}
\end{equation}
where $a = -\frac{1}{2} (1 + \sqrt{5})$. 
\end{example}
\begin{example} The category $\mathcal{E}$ associated with the even part of the $E_6$ subfactor \cite{hh09-snbfcor3, ostrik2013pivotal} has 3 simple objects: $1$ (drawn as a dotted line), $x$ (drawn as a solid line), $y$ (drawn as a wiggly line) and fusion rules $xy = yx =x$, $xx = 1 + 2x + y$, $yy=1$. So, the nontrivial trivalent basis elements are:
\[
\ig{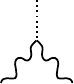}, \ig{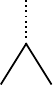}, \ig{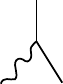}, \ig{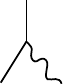} , \ig{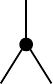}, \ig{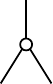} \, .
\]
Some relevant associators are
\begin{equation} \label{cate1}
\ig{d84.pdf} = \frac{1}{d} \left( \ig{d86.pdf} + \ig{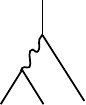} \right) + \frac{1}{\sqrt{2} v} \left( \ig{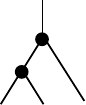} + \ig{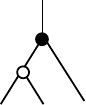} + \ig{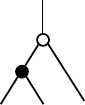} - \ig{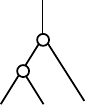} \right) 
\end{equation}
and:
\begin{align}
\ig{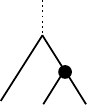} & = \frac{1}{\sqrt{2}} e^{\frac{-7 \pi i}{12}} \left( \ig{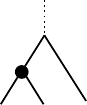} + \ig{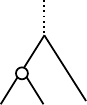}  \right) \label{cate2} \\
\ig{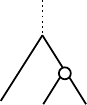} & = \frac{1}{\sqrt{2}} e^{\frac{-7 \pi i}{12}} \left( -i \ig{d102.pdf} + i \ig{d103.pdf} \right) \label{cate3}
\end{align}
Here $d = 1 + \sqrt{3}$ and $v = \sqrt{d}$. See \cite{hong2009onsymmetrization} for a full list. Note that we are using the inverses of the matrices in \cite{hong2009onsymmetrization} due to our conventions.
\end{example}

\section{The pivotal operators on a fusion category\label{stss}}
In this section we define the pairing convention and the pivotal operators on a fusion category. We prove that the paired dimensions are positive, and we compute the pivotal operators in terms of the apex associator monodromy and the pivotal indicators. We show that the pivotal operators are monoidal and identify them with the double dual functor.

\subsection{Fusion and Frobenius-Peron dimensions} \label{fusion_frobenius}

We recall the following from \cite{eno02-ofc}. In a fusion  category $\cat{C}$ over $\mathbb{C}$, every simple object $X_i$ has two canonical dimensions, which are positive real numbers: its {\em Frobenius-Perron} dimension $d^+_i$ and its {\em fusion} dimension $d_i$. 

\begin{defn}[see \cite{eno02-ofc}] The {\em Frobenius-Perron dimensions} $d^+_j$ of the simple objects $X_j$ are the unique positive real numbers satisfying
\begin{equation}
 d^+_j d^+_k = \sum_i N^i_{jk} d^+_i  \label{FPeqn}
\end{equation}
where $N^i_{jk} = \dim \Hom(X_i, X_j \otimes X_k)$. That is, they furnish the unique homomorphism from the Grothendieck ring of $\cat{C}$ to $\mathbb{C}$ taking positive real values on the simple objects $X_i$. 
\end{defn}
Note that, since taking duals is a ring anti-homomorphism from the Grothendieck ring of $\cat{C}$ to itself, it follows from uniqueness of the Frobenius-Perron dimensions that $d^+_j = d^+_{j^*}$ for all $j \in I$.

To define the fusion dimensions, we first need to define paired dimensions\footnote{The notion of a paired dimension is due to M\"{u}ger \cite[Prop 2.4]{m03-sct1}, where it is written as $d^2(X)$, but not named explicitly. The terminology and notation is our own.}.

\begin{defn} Let $X_i$ and $X_i^*$ be dual simple objects in $\cat{C}$. Their {\em paired dimension}   is
\begin{equation} \label{pair1}
  d_{\{i, i^*\}} = \ba \ig{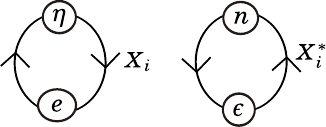} \ea,
 \end{equation}
where $(\eta, \epsilon)$ is some choice of unit and counit maps exhibiting $X_i^*$ as a right dual of $X_i$, and $(n, e)$ is some choice of unit and counit maps exhibiting $X_i^*$ as a {\em left} dual of $X_i$. 
\end{defn}
Observe that the product \eqref{pair1} is independent of the choices of unit and counit maps made in the definition because it is invariant under rescaling
 \[
 \eta \mapsto \lambda \eta, \epsilon \mapsto \frac{1}{\lambda} \epsilon, n \mapsto \mu n, e \mapsto \frac{1}{\mu} e.
\]
Note that the paired dimensions are certainly nonzero complex numbers. Indeed, $e \circ \eta$ is zero if
and only if one of $e$ or $\eta$ is zero
(by semisimplicity), which would contradict the duality equations \eqref{snake_eqns}.
Similarly for $\epsilon \circ n$. In Section \ref{piv_ind_sec} we will show how to compute the paired dimensions explicitly in terms of the associator matrix elements.

In fact, we will show in Theorem \ref{realpositive} that the paired dimensions are always positive real numbers. Anticipating this, we make the following definition.
\begin{defn}\label{fusion_dim}The {\em fusion dimension} $d_i$ of a simple object $X_i$ in a fusion category over $\mathbb{C}$ is the positive square root of its paired dimension $d_{\{i, i^*\}}$.
\end{defn}
In Corollary \ref{paired_cor}, we show how to read off the fusion dimensions directly from the associators.
\begin{example} In the Yang-Lee category, $d^+_\tau = \frac{1}{2}(1 + \sqrt{5})$ and we read off from \eqref{yl1} that $d_\tau = \frac{1}{2}(-1 + \sqrt{5})$. In category $\mathcal{E}$, $d^+_x = 1 + \sqrt{3}$ and $d^+_y = 1$, while we read off from \eqref{cate1} that $d_x = d^+_x$; similarly $d_y = d^+_y = 1$.
\end{example}
The {\em global dimension} $D_\cat{C}$ of a fusion category is the sum of its paired dimensions, $D_\cat{C} = \sum_{i \in I} d_{\{i, i^*\}}$. The {\em Frobenius-Perron dimension} $FP_\cat{C}$ of a fusion category is the sum of the squares of the Frobenius-Perron dimensions, $FP_\cat{C} = \sum_{i \in I} (d^+_i)^2$. The fusion category is called {\em pseudo-unitary} if $FP_\cat{C} = D_\cat{C}$. 

\subsection{The pairing convention} \label{pairing_sec}

\begin{defn} A {\em root choice} on a fusion category is a symmetric choice $\{d_i\}$ of square roots of the paired dimensions; that is, one which satisfies
 \[
  d_i^2 = d_{\{i,i^*\}} \text{ and } d_i = d_{i^*} \text{ for all $i \in I$.}
 \]
\end{defn}
The utility of a root choice is that it allows the following  diagrammatic convention.
\vskip 0.3cm
\framebox{\parbox[b]{11cm}{\textbf{Pairing convention.} Whenever a unit $\ba \ig{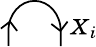} \ea$ which is part of a unit-counit pair expressing a simple object $X_i^*$ as a {\em right} dual of $X_i$ appears together in some equation with a counit $\ba \ig{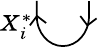} \ea$ which is part of a unit-counit pair expressing $X_i^*$ as a {\em left} dual of $X_i$, it will always be understood that the former is arbitrary while the latter is determined uniquely by the requirement that
 \[
  \ba \ig{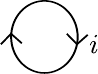} \ea = d_i 
 \]
Similarly (and equivalently), whenever $\ba \ig{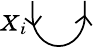} \ea$ appears in a diagram together with $\ba \ig{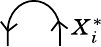} \ea$, the latter will always be fixed uniquely by the requirement that 
\[
\ba \ig{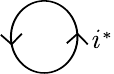} \ea = d_i.
\]}}
\vskip 0.3cm
The pairing convention is the key diagrammatic idea in this paper. Using this convention, we only need to fix a root choice (and not a pivotal structure, which is not known to exist in general) on the fusion category in order to unambiguously perform string diagram calculations where both left and right duals appear. Note that the pairing convention extends uniquely to the whole category and not just the simple objects, see Section \ref{piv_monoidal}.

\begin{rem}
We will presently show in Theorem \ref{realpositive} that the paired dimensions $d_{\{i, i^*\}}$ are positive real numbers. From then on, we will always work with the {\em canonical} root choice on a fusion category given by the positive square roots.
\end{rem}

Although we will not need it in this paper, we can extend the pairing convention to {\em all} objects in $\cat{C}$, not just simple ones. Namely, suppose $V \in \cat{C}$. For each representative simple object $X_i$, choose a basis $e_{i, \alpha} : X_i \rightarrow V$ for $\Hom(X_i, V)$, with corresponding dual basis $\hat{e}_{i, \alpha} : V \rightarrow X_i$. If, in a diagram, a unit
\[
\ig{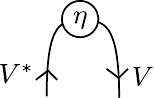} 
\] 
which is part of a unit-counit pair $(\eta, \epsilon)$ expressing $V^*$ as a {\em right} dual of $V$ appears together with a counit
\[
\ig{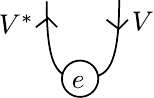}
\]
which is part of a unit-counit pair $(n, e)$ expressing $V^*$ as a {\em left} dual of $V$, it will be understood that the former is arbitrary while the latter is defined as
\[
\ba \ig{h15.pdf} \ea := \sum_{i, \alpha} \ba \ig{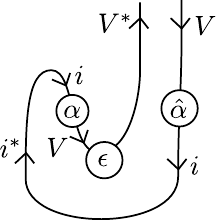} \ea \, .
\]
Note that this definition is independent of the choice of basis $e_{i, \alpha}$. 

\subsubsection{Further graphical conventions}

Fix a root choice $\{d_i\}$, $i \in I$. We introduce the following notation. For an endomorphism of a simple object $f : X_i \rightarrow X_i$, we will write $\langle f \rangle \in \mathbb{C}$ for the scalar satisfying $f = \langle f \rangle \, \id_{X_i}$. In string diagrams, 
\begin{equation} \label{extract_scalar}
 \ba \ig{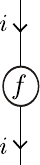} \ea \, = \, \langle f \rangle \, \ba \ig{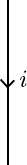} \ea  \, .
\end{equation}
We can extract the scalar $\langle f \rangle$ by pre- and post-composing both sides with the cup and cap maps, using the pairing convention with respect to the root choice. This gives
\begin{equation} \label{loop_scalar}
 \langle f \rangle = \frac{1}{d_i} \, \ba \ig{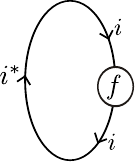} \ea \, .
\end{equation}
Note that since $d_{i} = d_{i^*}$, it doesn't matter in which direction one closes the loop. 

Similarly, given a morphism $f : X_i \rightarrow X_j \otimes X_k$, we can insert the identity on $X_j \otimes X_k$ in the form of \eqref{res_id}, and close the loop as above, to expand $f$ in a basis $e_\beta : X_i \rightarrow X_k \otimes X_k$ as follows:
\begin{equation} \label{expand_loop}
\ba \ig{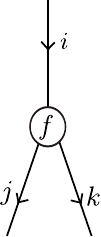} \ea \, = \frac{1}{d_i} \sum_\beta \ba \ig{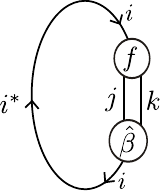} \ea \, \ba \ig{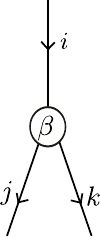} \ea
\end{equation}

\subsection{The pivotal operators} \label{piv_op_sec}
Using the pairing convention, we can diagrammatically define involutions on the fusion vector spaces.
\begin{defn} Let $\{d_i\}$ be a root choice on $\cat{C}$. The {\em pivotal operators} 
\[
T^i_{jk} \colon \Hom(X_i, X_j \otimes X_k) \rightarrow \Hom(X_i, X_j \otimes X_k)
\] 
are defined by
 \begin{equation} \label{defn_piv_operator}
  \ba \ig{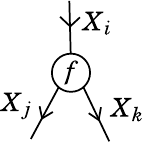} \ea \mapsto \ba \ig{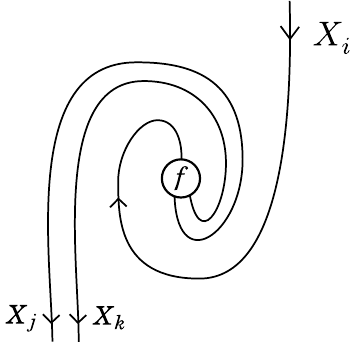} \ea.
 \end{equation}
where the pairing convention has been used.
\end{defn}
\noindent That is, to evaluate the right hand side of \eqref{defn_piv_operator}, make an arbitrary choice of right duals $(X_i^*, \eta_i, \epsilon_i)$, $(X_j^*, \eta_j, \eta_j)$ and $(X_k^*, \eta_k, \epsilon_k)$ for $X_i$, $X_j$ and $X_k$, and then choose the left dual structure maps using the pairing convention. 

Note that changing the root choice by setting $d_i \mapsto x_i d_i$ for some signs $x_i = \pm 1$ with $x_i = x_{i^*}$ will change the sign of the pivotal operators according to $T^i_{jk} \mapsto x_ix_jx_k T^i_{jk}$.

\begin{lem} \label{mat_elements} In an arbitrary trivalent basis $\{e_\alpha \colon X_i \rightarrow X_j \otimes X_k\}$ with corresponding dual basis $\{\hat{e}_\beta \colon X_j \otimes X_k \rightarrow X_i\}$ in the sense of \eqref{dual_basis}, the matrix elements of $T^i_{jk}$ compute as
 \[
 \langle \hat{e}_\beta T^i_{jk} (e_\alpha) \rangle = \frac{1}{d_i} \, \ba \ig{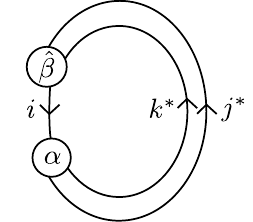} \ea \, .
 \]
\end{lem}
\begin{proof} We can compute the scalar $ \langle \hat{e}_\beta T^i_{jk} (e_\alpha) \rangle$ by closing the loop, as in \eqref{loop_scalar}:
 \[
 \hat{e}_\beta T^i_{jk}  e_\alpha = \frac{1}{\ba \ig{e288.pdf} \ea} \ba \ig{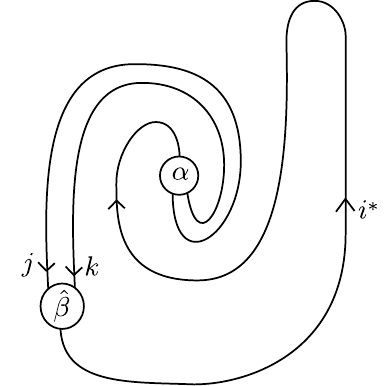} \ea = \frac{1}{d_i} \, \ba \ig{e295.pdf} \ea \, .
 \]
\end{proof}

\begin{lem} Suppose $X_i, X_j$ are simple objects and $f : X_i \otimes X_j \rightarrow X_i \otimes X_j$. Then the following equations hold in the pairing convention: \label{flip_lem}
\[
 \ba \ig{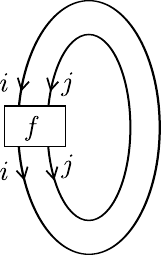} \ea \stackrel{\mathrm{(a)}}{=} \ba \ig{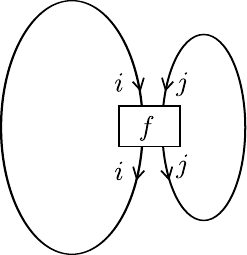} \ea \stackrel{\mathrm{(b)}}{=} \ba \ig{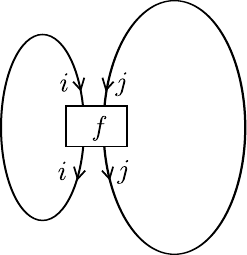} \ea \stackrel{\mathrm{(c)}}{=} \ba \ig{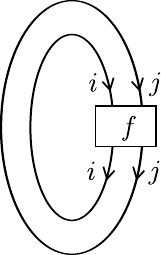} \ea \,.
 \]
\end{lem}
\begin{proof}
To prove (a), write 
\[
 \ba \ig{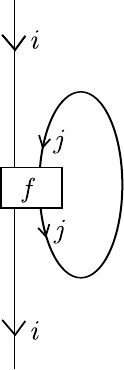} \ea \, = \, \lambda \, \ba \ig{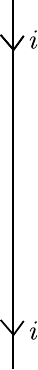} \ea 
\]
for some scalar $\lambda$, as in equation \eqref{extract_scalar}, and extract the value of $\lambda$ as in equation \eqref{loop_scalar}. Inserting this back into the left hand side of (a) immediately gives the right hand side. Equation (b) is the interchange law in a monoidal category. The proof of (c) is similar to the proof of (a), but uses in addition $d_i = d_{i^*}$.
 \end{proof}
 

\begin{thm}[{cf. \cite[Thm 3]{hh09-snbfcor3}}]  The operator $T^i_{jk}$ is an involution --- that is, $(T^i_{jk})^2 = \id$. \label{piv_involution_thm}
\end{thm}
\begin{proof} The operator $T^2$ sends
 \[
  \ba \ig{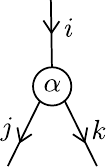} \ea \mapsto \ba \ig{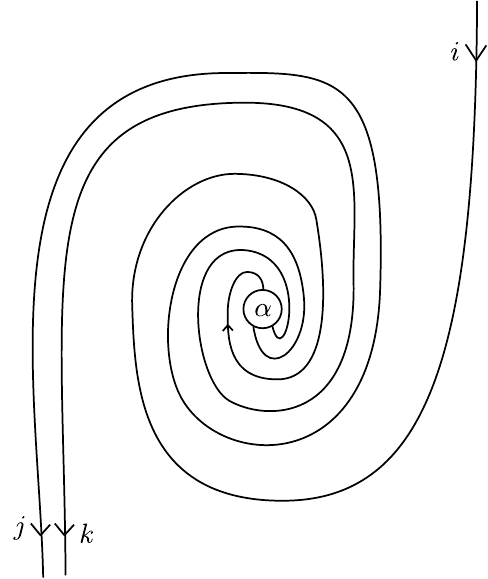} \ea .
 \]
Its matrix elements are thus:
\[
\ba \ig{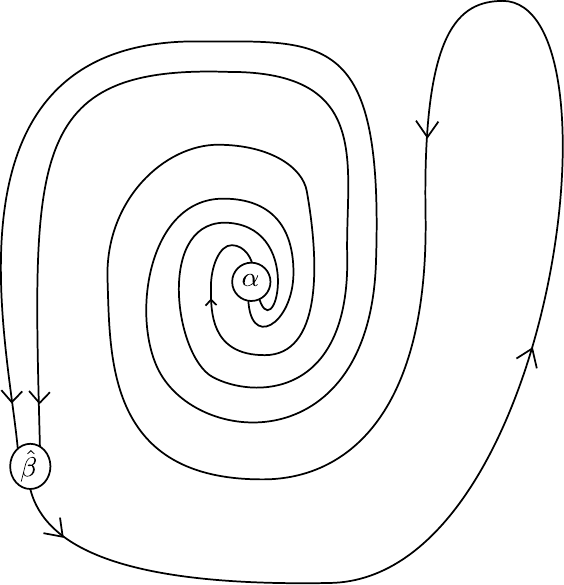} \ea = \ba \ig{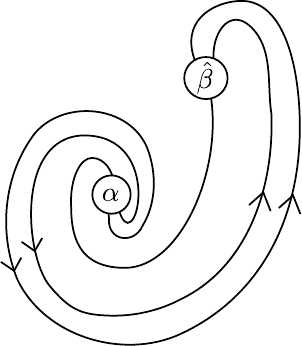} \ea
\]
\[
= \ba \ig{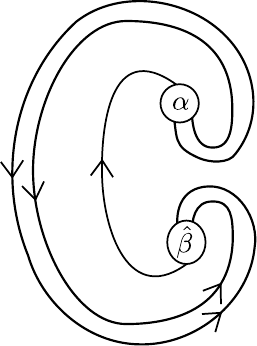} \ea = \ba \ig{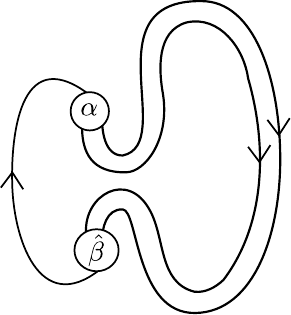} \ea = \ba \ig{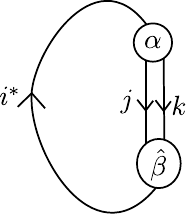} \ea = d_i \delta^\beta_\alpha
\]
where we have used Lemma \ref{flip_lem} twice in the third equation. In other words, we have expanded $T^2(e_\alpha)$ in the basis $e_\alpha$ using the technique from eq. \eqref{expand_loop}, and we have shown that $T^2 (e_\alpha) = e_\alpha$.  
\end{proof} \noindent
\begin{rem}This calculation is essentially the {\em Dirac belt trick} proving that $\pi_1 SO(3) = \mathbb{Z}/ 2\mathbb{Z}$, and provides a link between fusion categories and spin structures, as mentioned in Remark \ref{spin_remark}. See also \cite{Douglas:2013aa} and Remark \ref{remark_2_cocycle}.
\end{rem}
\begin{rem} If $\cat{C} = \text{Rep} H$ for a semisimple Hopf algebra $H$, then the identity
$(T^i_{jk})^2 = \id$ corresponds to the Larson-Radford formula $S^2 = \id$ (see \cite{eno02-ofc} and references therein). The string diagram argument above can be regarded
as giving a graphical proof of Radford's formula.
\end{rem}
\begin{rem} The proof in \cite[Thm 3]{hh09-snbfcor3} proceeds by first passing to a strictified skeletal category equivalent to $\cat{C}$. In our approach, the pairing convention, together with our conventions on string diagrams from Section \ref{string_conventions}, allows us to work directly in the category $\cat{C}$.
\end{rem}
Since $T^i_{jk}$ is an involution, the vector space $\Hom(X_i, X_j \otimes X_k)$ has a basis of eigenvectors $e_\alpha$ whose eigenvalues are $\pm 1$. We call such a basis a {\em pivotal basis} for the fusion category $\cat{C}$.

\begin{defn} The {\em pivotal symbols} $\epsilon^i_{jk,\alpha} = \pm 1$ are the eigenvalues of the pivotal operators,  $T^i_{jk} e_\alpha = \epsilon^i_{jk, p} e_\alpha$.
\end{defn}
Thus, in a fusion category, the vector spaces $\Hom(X_i, \, X_j \otimes X_k)$ decompose into the positive and negative eigenspaces of $T^i_{jk}$, \begin{equation} \label{decomp_eqn}
\Hom(X_i, \, X_j \otimes X_k) = \Hom(X_i, \, X_j \otimes X_k)_+ \,\oplus \, \Hom(X_i, \, X_j \otimes X_k)_-
\end{equation}
Note that this decomposition is {\em canonical} since our convention will be to make the unique root choice where the $d_i$ are positive real numbers, as guaranteed by Theorem \ref{realpositive} below.
\begin{lem}\label{discomb} In a pivotal basis, the pivotal symbols can be computed as follows:
 \[
 \epsilon^i_{jk, \alpha} = \frac{1}{d_i}
  \ig{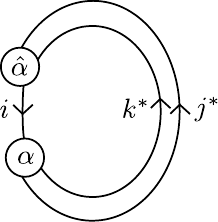} \, .
  \]
\end{lem}
\begin{proof}
Follows immediately from Lemma \ref{mat_elements}.
\end{proof}
\begin{lem} The pivotal operators $T^1_{11}$ and $T^1_{i i^*}$ are the identity maps for all $i \in I$. \label{triv_pivs}
\end{lem}
\begin{proof} Follows from an elementary string diagram argument.
\end{proof}

\subsection{Positivity of paired dimensions}In this subsection, we show that the paired dimensions $d_{\{i,i^*\}}$ are positive real numbers.

The first step is to show that a root choice gives rise to a `twisted' homomorphism from the Grothendieck ring to $\mathbb{C}$, which is to be compared with \eqref{FPeqn}.

\begin{prop}[{cf. \cite[pg. 593]{eno02-ofc}}]\label{homprop} In any root choice $\{d_i\}$, we have 
\[
d_j d_k = \displaystyle \sum_i \Tr (T^i_{jk}) \, d_i.
\]
\end{prop}
\begin{proof} \begin{align*}
 d_j d_k &= \ba \ig{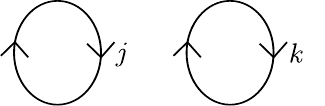} \ea \\
  &= \ba \ig{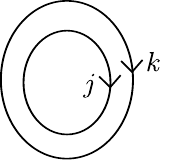} \ea \\
  &= \sum_{i,\alpha} \ba \ig{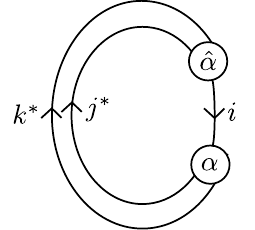} \ea \\
 &= \sum_{i,\alpha} \epsilon^i_{jk,\alpha} d_i \quad \text{(by Lemma \ref{discomb})} \\
 &= \sum_i \Tr(T^i_{jk}) \, d_i.
  \end{align*}
\end{proof}

\begin{lem}[{cf. \cite[pg 594]{eno02-ofc}}]\label{Tsymprop} The numbers $\Tr(T^i_{jk})$ have the following symmetry properties:
 \begin{enumerate}
  \item $\Tr(T^i_{jk}) = \Tr(T^{k^*}_{i^*j}) \quad$ \text{(Conjugate cyclic)} \\
  \item $\Tr(T^i_{jk}) = \Tr(T^{i^*}_{k^*j^*}) \quad$ \text{(Conjugate symmetric)}
 \end{enumerate}
 \end{lem}
 \begin{proof} To establish (i), suppose that $\{ e_\alpha \colon X_i \rightarrow X_j \otimes X_k \}$ are eigenvectors of $T^i_{jk}$, so that $T^i_{jk} e_\alpha = \epsilon^i_{jk, \alpha} e_\alpha$. Choose arbitrary right duals $(X_i^*, \eta_i, \epsilon_i)$ and $(X_k^*, \eta_k, \epsilon_k)$ for $X_i$ and $X_k$. Then the basis for $\Hom(X_k^*, X_i^*, X_j)$ given by
  \[
   f_\alpha := \ba \ig{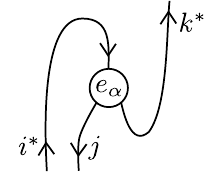} \ea
  \]
are also eigenvectors of $T^{k^*}_{i^*, j}$ with the same eigenvalues as the $e_\alpha$, since
  \[
  T^{k^*}_{i^*, j} (f_\alpha) = \ba \ig{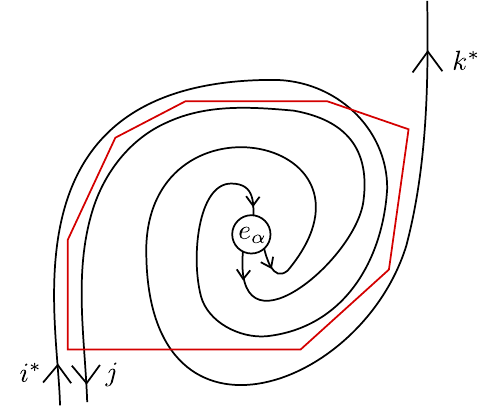} \ea = \epsilon^i_{jk, \alpha} \ba \ig{e216.pdf} \ea = \epsilon^i_{jk,\alpha} f_\alpha.
  \]
The proof of (ii) is similar, except one uses the {\em dual} basis $\{\hat{e}_\alpha \colon X_j \otimes X_k \rightarrow X_i\}$ to define a basis $\{g_\alpha\}$ for $\Hom(X_i^*, X_k^* \otimes X_j^*)$ by setting
 \[
  g_\alpha := \ba \ig{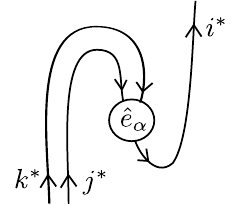} \ea.
 \]
A similar string diagram argument then establishes that $T^{i^*}_{k^*j^*} g_\alpha = \epsilon^i_{jk, \alpha} g_\alpha$.
 \end{proof}
\begin{thm}[{cf. \cite[Thm 2.3]{eno02-ofc}}]\label{realpositive} The paired dimensions $d_{\{i, i^*\}}$ are real and positive.
\end{thm}
\begin{proof} We organize the various roots $d_i \equiv \ba \ig{e288.pdf} \ea$ into a column vector:
 \[
   \mathbf{d} = \left( \ba \ig{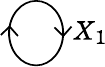} \ea, \ba \ig{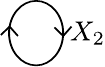} \ea, \ldots, \ba \ig{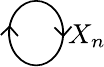} \ea \right)^T.
 \]
Define the matrices $A_j$ via $[A_j]_{ik} = \Tr(T^k_{ji})$. Then Proposition \ref{homprop} says that $\mathbf{d}$ is a simultaneous eigenvector of each $A_j$ with eigenvalue $d_j$, i.e.
 \[
  A_j \mathbf{d} = d_j \mathbf{d}.
 \]
Thus we have $A_j A_{j^*} \mathbf{d} = d_j d_{j^*} \mathbf{d} = d_{\{j,j^*\}} \mathbf{d}$ and so $d_{\{j,j^*\}}$ is an eigenvalue of the matrix $A_j A_{j^*}$. But $A_{j^*} = A_j^T$, by the symmetry properties established in Lemma \ref{Tsymprop}:
 \begin{align*}
  [A_{j^*}]_{ik} &= \Tr(T^k_{j^* i}) \\
   &= \Tr(T^{i^*}_{k^* j^*})  \\
   &= \Tr(T^i_{jk}) \\
   &= [A_j]_{ki}.
  \end{align*}
Thus $d_{\{j,j^*\}}$ is an eigenvalue of the positive definite real matrix $A_j A_j^T$ and is therefore real and positive.
\end{proof}
\begin{rem} From now on, we will always work in the {\em canonical} root choice on a fusion category given by the positive square roots of the paired dimensions, $d_i = \sqrt{d_{\{i, i^*\}}}$. 
\end{rem}
\begin{cor}[{cf. \cite[Prop 8.21]{eno02-ofc}}] A fusion category is pseudo-unitary if and only if its Frobenius-Perron dimensions equal its fusion dimensions, that is $d^+_i = d_i$ for all $i \in I$. \label{char_pseudo_cor}
\end{cor}
\begin{proof} It suffices to show that $d_i \leq d^+_i$ for all $i \in I$. Use the Frobenius-Perron dimensions to define a norm on $\mathbb{C}^{|I|}$ by setting $\| \mathbf{x} \| = \sum_k d^+_k |x_k|$. Then we can estimate the operator norm of the operators $A_j$ from the proof of Theorem \ref{realpositive} as follows:
\begin{eqnarray*}
 \| A_j \mathbf{x} \| &=& \sum_{k} d^+_k \left| \sum_i \Tr(T^i_{jk}) x_i \right| \\
	& \leq & \sum_{k, i} d^+_k  N^i_{jk} |x_i| \\
	& = & \sum_{i,k} d^+_k N^k_{j i^*} |x_i| \\
	& = & d^+_j \sum_i d^+_{i^*} |x_i| \\
	& = & d^+_j \|x\| \, .
\end{eqnarray*}
In the second line we used $\Tr(T^i_{jk}) \leq N^i_{jk}$, in the third line the symmetry $N^i_{jk} = N^k_{ji^*}$, in the fourth line we used the definition of the Frobenius-Perron dimensions, and in the fifth line we used $d^+_{i^*} = d^+_{i}$.
\end{proof}
\begin{cor} In a pseudo-unitary fusion category, the pivotal operators $T^i_{jk}$ are the identity maps. \label{pseudo_unitary}
\end{cor}
\begin{proof} Combining eq. \eqref{FPeqn} with Proposition \ref{homprop} and Corollary \ref{char_pseudo_cor}, we have
\[
d_j d_k = \sum_i \Tr(T^i_{jk}) d_i \leq \sum_i N^i_{jk} d_i = d_j d_k
\]
and hence we must have $T^i_{jk} = \id$.
\end{proof}

\subsection{The pivotal operators as a monoidal functor} \label{piv_monoidal}
We want to precisely relate the pivotal operators with the more traditional notion of the `double dual functor'. We first need to show that the pivotal operators can be viewed as equipping the identity functor with coherence isomorphisms making it into a monoidal functor. 

\begin{lem} \label{extend}Suppose $F,G : \cat{D} \rightarrow \cat{D}'$ are linear functors between semisimple categories. Let $\theta_{i} : F(X_i) \rightarrow G(X_i)$ be a collection of maps, where $X_i$ ranges over the representatives of the simple objects of $\cat{D}$. Then there exists a unique natural transformation $\theta : F \Rightarrow G$ such that $\theta_{X_i} = \theta_i$.
\end{lem}
\begin{proof} Let $A \in \cat{D}$, and for each $i$, choose a basis $\{e_\alpha : X_i \rightarrow A\}$ for $\Hom(X_i, A)$ with corresponding dual basis $\{e^\beta : A \rightarrow X_i\}$. Set  $\theta_A = \sum_{i,\alpha} G(e_\alpha) \circ \theta_i \circ F(e^\alpha)$. The rest of the proof (naturality of $\theta$ and its uniqueness) is the same as \cite[Lemma 3.9]{BBthesis}.
\end{proof}
This allows us to extend the pivotal operators $T^i_{jk}$, initially defined only on the simple objects, to a natural transformation $T \equiv \{T_{A,B} :\ A \otimes B \rightarrow A \otimes B$\} where $A$ and $B$ are arbitrary objects of $\cat{C}$.
\begin{cor} The pivotal operators $T^i_{jk}$ extend uniquely to a natural transformation $T : \otimes \Rightarrow \otimes$ from the tensor functor $\otimes : \cat{C} \boxtimes \cat{C} \rightarrow \cat{C}$ to itself. \label{extend_cor}
\end{cor}
\begin{proof}
The pivotal operators $T^i_{jk} : \Hom(X_i, X_j \otimes X_k) \rightarrow \Hom(X_i, X_j \otimes X_k)$ induce maps
\[
T_{j,k} : X_j \otimes X_k \rightarrow X_j \otimes X_k
\]
for each pair of simple objects $X_j$, $X_k$, defined uniquely by the requirement that for each simple object $X_i$, the map on the hom-sets given by post-composing with $T_{j,k}$,
\begin{align*}
 \text{post}(T_{j,k}) : \Hom(X_i, X_j \otimes X_k) & \rightarrow \Hom(X_i, X_j \otimes X_k)  \\
 f & \mapsto T_{j,k} \circ f
\end{align*}
is equal to $T^i_{jk}$. That is, we require $\text{post}(T_{j,k}) = T^i_{j,k}$ as endomorphisms of the vector space $\Hom(X_i, X_j \otimes X_k)$. By the Yoneda lemma, this indeed gives unique well-defined maps $T_{j,k}$. Now apply Lemma \ref{extend} to the case $\cat{D} =\cat{C} \boxtimes \cat{C}$, $\cat{D}' = \cat{C}$, and $F$($=G$) given by the tensor functor $\otimes : \cat{C} \boxtimes \cat{C} \rightarrow \cat{C}$.  
\end{proof}
Recall that a monoidal functor $(F, T, \phi) : \cat{C} \rightarrow \cat{D}$ \label{monoidal_page}  between monoidal categories $\cat{C}$ and $\cat{D}$ consists of a functor $F :\ \cat{C}\ \rightarrow \cat{D}$ together with natural isomorphisms $\{T_{A,B} : F(A) \otimes F(B) \rightarrow F(A \otimes B), \, A, B \in \cat{C}\}$ and $\phi : 1_{\cat{D}} \rightarrow F(1_{\cat{C}})$ satisfying certain coherence equations.
\begin{prop} The pivotal operators $T^i_{jk}$, when extended to a natural transformation $\{T_{A,B} : A \otimes B \rightarrow A \otimes B\}$, obey the coherence isomorphisms making $T := (\id, T, \id) :\ \cat{C} \rightarrow \cat{C}$ a monoidal functor.
\end{prop}

\begin{proof} By Lemma \ref{extend}, we only need to check the coherence equations on the simple objects. The unit coherence equations are $T^1_{i,1} = T^1_{1,i} = \id$, which is automatically satisfied due to the rigidity equations \eqref{snake_eqns}. The coherence conditions on $T$ are:
\begin{equation} \label{coherence_on_T}
\begin{tikzpicture}[xscale=6, yscale=1.5]
\node (1) at (0,0) {$ (F(X_j) \otimes F(X_k) ) \otimes F(X_l)$};
\node (2) at (0.2,1) {$F(X_j \otimes X_k) \otimes F(X_l)$};
\node (3) at (1.2,1) {$F((X_j \otimes X_k) \otimes X_l)$};
\node (4) at (1.4,0) {$F(X_j \otimes (X_k \otimes X_l))$};
\node (5) at (0.2,-1) {$F(X_j) \otimes (F(X_k) \otimes F(X_l))$};
\node (6) at (1.2,-1) {$F(X_j) \otimes F(X_k \otimes X_l)$};
\draw[->] (1) to node[left] {$T_{X_j, X_k} \otimes \id$} (2); 
\draw[->] (2) to node[above] {$T_{X_j \otimes X_k, X_l}$} (3); 
\draw[->] (3) to node[right] {$F(a_{X_j, X_k, X_l})$} (4); 
\draw[->] (1) to node[left] {$a_{F(X_j), F(X_k), F(X_l)}$} (5); 
\draw[->] (5) to node[below] {$\id \otimes T_{X_k, X_l}$} (6); 
\draw[->] (6) to node[right] {$T_{X_j, X_k \otimes X_l}$} (4); 
\end{tikzpicture} 
\end{equation}
Apply this diagram to the basis vectors \eqref{basis_choice}, working in a pivotal basis to simplify the calculation. It becomes the requirement that
\begin{equation}
\ba \ig{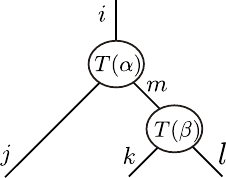} \ea = \sum_{n, \gamma, \delta} (F^i_{jkl})^{\alpha m \beta}_{\gamma n \delta} \, \ba \ig{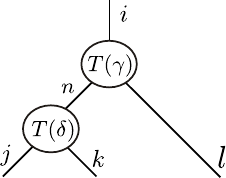} \ea \label{monoidal_pivotal}
\end{equation}
To prove \eqref{monoidal_pivotal}, insert the graphical definition \eqref{defn_piv_operator} of the pivotal operators into \eqref{monoidal_pivotal}, and expand out the left hand side using the associator expansion \eqref{associator_eqn}. As in the proof of Theorem \ref{piv_involution_thm}, we obtain two diagrams which are rigidly isotopic, hence equal.  
\end{proof}
\begin{cor} In a pivotal basis, the pivotal symbols satisfy
\begin{equation} \label{2-cocycle} 
 \epsilon^i_{jk, \alpha} \epsilon^m_{kl, \beta} = \epsilon^i_{nl, \gamma} \epsilon^n_{jk, \delta} \quad \text{whenever } (F^i_{jkl})^{\gamma n \delta}_{\alpha m \beta} \neq 0.
\end{equation}
\end{cor}

\begin{proof}
This is precisely what \eqref{monoidal_pivotal} says, in a pivotal basis.
\end{proof}
\begin{rem} The equation \eqref{2-cocycle} satisfied by the $\epsilon^i_{jk, \alpha}$ is formally similar to the equation satisfied by the 2nd Stiefel-Whitney 2-cocycle $w_2(M) \in H^2(M, \mathbb{Z}/2\mathbb{Z})$ of an oriented manifold $M$ in \v{C}ech cohomology \cite{naber-2011}. It can be thought of as a `$\mathbb{Z}/2$ version' of the 2-cocycle condition in Davydov-Yetter cohomology of the fusion category, see Example \ref{2-cocycle-example}. \label{remark_2_cocycle}
\end{rem}

Let us summarize the results of this section explicitly as follows.

\begin{thm}[{cf. \cite[Thm 2.6]{eno02-ofc}}] Every fusion category $\cat{C}$ over $\mathbb{C}$ comes equipped with a canonical monoidal endofunctor $\mathcal{T} : C \rightarrow C$, the pivotal endofunctor, satisfying:
\begin{itemize}
 \item The underlying functor of $\mathcal{T}$ is the identity functor,
 \item The monoidal coherence isomorphisms $T : A \otimes B \rightarrow A \otimes B$ of $\mathcal{T}$ are specified by the pivotal operators $T^i_{jk}$, extended to a natural transformation using Corollary \ref{extend_cor},
 \item $\mathcal{T}^2 = \id$. 
\end{itemize}
 \label{piv_operator_C}
\end{thm}

\subsection{The pivotal operators are the double dual functor} \label{pivops}
Our approach to the `double dual functor' on a fusion category $\cat{C}$ been to encode it in a canonical monoidal endofunctor $\mathcal{T} : \cat{C} \rightarrow \cat{C}$, whose underlying functor is simply the identity. The advantage of our approach is that it is canonical, not requiring arbitrary choices of duals for objects. Such fixed initial choices of duals tend to complicate calculations, as the expression for $\omega_{V,W}$ in equation \eqref{defn_omega} below shows.  We now connect our approach to the usual approach, taking more care than is common in the literature. 

We define a {\em system of right duals} on a fusion category $\cat{C}$ as a choice, for every object $V \in \cat{C}$, of a triple
\[
(V^*, \, \eta_V : 1 \rightarrow V^* \otimes V, \, \epsilon_V : V \otimes V^* \rightarrow 1)
\]
which equips $V^*$ as a right dual of $V$. It is then a routine calculation to check that this makes $* : \cat{C} \rightarrow \cat{C}^\text{op}$ into a well-defined functor, which is equipped with canonical isomorphisms $V^* \otimes W^* \rightarrow (W \otimes V)^*$ making it into a monoidal functor. In particular, given such a choice of right duals, we have the associated double dual monoidal endofunctor $**$ of $\cat{C}$, and the following explicit coherence isomorphisms $\omega_{V, W} : V^{**} \otimes W^{**} \rightarrow (V \otimes W)^{**}$, expressed graphically as follows:
\begin{equation} \label{defn_omega}
\omega_{V,W} = \ba \ig{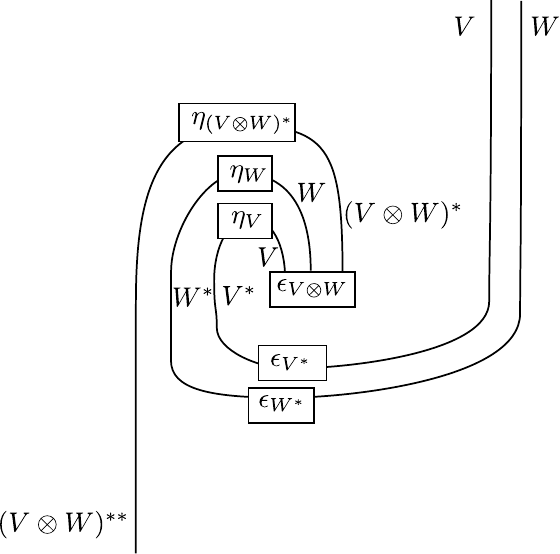} \ea 
\end{equation}
Given any other system of right duals $*'$ on $\cat{C}$, there is a unique monoidal natural isomorphism $\vartheta : * \Rightarrow *'$ such that $\eta_{V^{*'}} = (\vartheta \otimes \id_V) \circ \eta_{V^*}$. Horizontally composing $\vartheta$ with itself, we get a monoidal natural isomorphism $\theta : ** \Rightarrow *' *'$.

\begin{thm} \begin{enumerate} 
\item Given any system of right duals $*$ on a fusion category $\cat{C}$, there is a canonical monoidal natural isomorphism $\beta: \mathcal{T} \Rightarrow **$.
\item These isomorphisms are natural in the following sense: if $*'$ is another system of right duals on $\cat{C}$, then the diagram
\[
\begin{tikzpicture}[scale=2]
 \node (1) at (0,0) {$V$};
 \node (2) at (0, -1) {$V^{**}$};
 \node (3) at (1, -1) {$V^{*' *'}$};
 \draw[->] (1) -- node[left] {$\beta^{(*)}_V$} (2);
 \draw[->] (1) -- node[above right] {$\beta_V^{(*')}$} (3);
 \draw[->] (2) -- node[below] {$\theta_V$} (3); 
\end{tikzpicture}
\]
commutes for all $V \in \cat{C}$.
\end{enumerate}
\end{thm}
\begin{proof}
(i) We define $\beta_V$ by choosing a basis $e_{i, \alpha} : X_i \rightarrow V$ for the hom-sets $\Hom(X_i, V)$ where $i$ ranges over a representative set of simple objects, with corresponding dual basis $\hat{e}_{i, \alpha}$, and then setting 
\begin{equation} \label{defn_beta}
\beta_V := \sum_{i, \alpha} \ba \ig{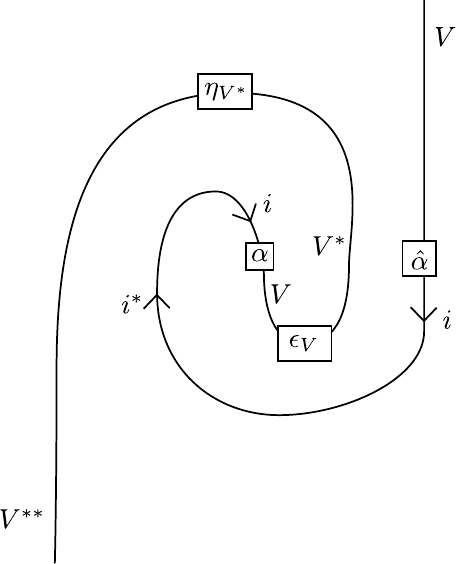} \ea 
\end{equation}
where we have used the partner convention on the $X_i$ strands. It is easy to check that this definition is independent of the basis choice $e_{i, \alpha}$, and that this indeed defines a natural transformation $\beta : \mathcal{T} \Rightarrow **$ of the underlying functors. For $\beta$ to be monoidal, we need to verify that the following coherence diagram commutes:
\begin{equation} \label{comm_check_diagram}
\begin{tikzpicture}[xscale=3, yscale=2]
\node (1) at (0,0) {$V \otimes W$};
\node (2) at (1,0) {$V \otimes W$};
\node (3) at (0, -1) {$(V \otimes W)^{**}$};
\node (4) at (1, -1) {$V^{**} \otimes W^{**}$};
\draw [->] (1) -- node[above] {$T_{V,W}$} (2);
\draw [->] (1) -- node[left] {$\beta_V \otimes \beta_W$} (3);
\draw [->] (3) -- node[below] {$\omega_{V, W}$} (4);
\draw [->] (2) -- node[right] {$\beta_{V \otimes W}$} (4);
\end{tikzpicture}
\end{equation}
Substitute the definition \eqref{defn_omega} of $\omega_{V, W}$ and the definition \eqref{defn_beta} of $\beta_V$ into the equation \eqref{comm_check_diagram}. The left and right hand sides simplify considerably. It is sufficient to check the remaining identity on the representative simple objects $V = X_k$, $W = X_l$, where it is an elementary string diagram verification, best performed in a pivotal basis. The diagram for part (ii) is simpler and can be verified directly.  
\end{proof}
Combining this with Theorem \ref{piv_operator_C}, which says that $\mathcal{T}^2 = \id$, we immediately obtain the following.
\begin{cor}[{cf. \cite[Thm 2.6]{eno02-ofc}}] For any system of right duals $*$ on a fusion category $\cat{C}$, there is a canonical monoidal natural isomorphism $\id \Rightarrow {*}{*}{*}{*}$, which is natural with respect to the choice of system of right duals $*$. \label{explicit_eno_cor}
\end{cor}

\section{Explicit formula for the pivotal operators} \label{explicit_sec}
In this section we express the pivotal operators $T^i_{jk}$ directly in terms of the associators of the fusion category. More precisely, we introduce the {\em apex-associator monodromy} and the {\em pivotal indicators}, and in Theorem \ref{formula_piv} we express $T^i_{jk}$ as the product of the pivotal indicators with the conjugate of the apex-associator monodromy.

\subsection{Monodromy interpretation of the pivotal operators}
We can write the pivotal operators $T^i_{jk}$ as the conjugate of a product of {\em cyclic operators}, defined as follows. Write $V_{ijk} = \Hom(1, X_i \otimes (X_j \otimes X_k))$.
\begin{defn} In a fusion category, the {\em cyclic operators} $C_{ijk} : V_{ijk} \rightarrow V_{kij}$ are defined as follows, using the pairing convention: \label{cyclic_op_defn}
\[
 \ba \ig{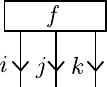} \ea \mapsto \ba \ig{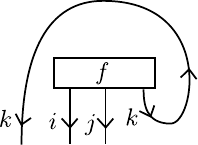} \ea 
\]
The {\em monodromy} of the cyclic operators is the operator $T_{ijk}$ on $V_{ijk}$ given as the composite
\[
  T_{ijk} = V_{ijk} \stackrel{C_{ijk}}{\longrightarrow} V_{kij} \stackrel{C_{kij}}{\longrightarrow} V_{jki}  \stackrel{C_{jki}}{\longrightarrow} V_{ijk} \, .
\]
\end{defn} 
Write $Y$ and $Y^{-1}$ for the `yanking' maps
\[
 Y : \Hom(X_i, X_j \otimes X_k) \rightleftarrows \Hom(1, X_i^* \otimes X_j \otimes X_k)
\]
defined by making some choice of structure maps equipping $X_i^*$ as the right dual of $X_i$. The following lemma is immediate.
\begin{lem} The pivotal operator $T^i_{jk}$ computes as the conjugate of the monodromy of the cyclic operators: \label{piv_op_lem_form}
\[
T^i_{jk} = Y^{-1} T_{i^* j k} Y 
\]
\end{lem}
Note that the statement of the lemma is well-defined since the right-hand side is independent of the right dual structure maps needed to define $Y$ (the dependency cancels due to the $Y^{-1}$).

\begin{rem} In the special case  $i=j=k$, the trace of the cyclic operator $C_{iii}$ is closely related, but not equal, to the 3rd Frobenius-Schur indicator $\nu_3(X_i)$ of $X_i$, as defined in \cite{ng2005higher}. The indicator $\nu_3(X_i)$ needs a pivotal structure $\gamma$ for its definition, whence we will write it as $\nu_3^\gamma(X_i)$, whereas the cyclic operator $C_{iii}$ is well-defined on the underlying fusion category. The 3rd Frobenius-Schur indicator $\nu_3^\gamma(X_i)$ is defined as the trace of an operator whose graphical formula is identical to the inverse of the formula for $C_{iii}$ in Definition \ref{cyclic_op_defn}, except it is to be interpreted using the pivotal structure $\gamma$ instead of the pairing convention. Hence on simple objects $X_i$ the relationship is 
\[
\nu_3^\gamma (X_i) = \gamma_i \Tr(C_{iii})
\]
where $\gamma_i \in U(1)$ are the scalars defining the pivotal structure (see Section \ref{Explicit_equations}). 
 \label{remark_frobenius}
\end{rem}

\subsection{The pivotal indicators \label{piv_ind_sec}} 
In this subsection we compute the paired dimensions explicitly and define the {\em pivotal indicators}. 

Let $\cat{C}$ be a fusion category, with representative simple objects $X_i$, $i \in I$. The right duality induces a map $* : I \rightarrow I$ with $**=\id$. Choose for each $i \in I$ a nonzero map $\eta_i : 1 \rightarrow X_{i^*} \otimes X_i$:
\[
\ig{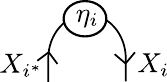}
\]
Define $\hat{\eta_i} : X_{i^*} \otimes X_i \rightarrow 1$ as the linear dual of $\eta_i$, so that
\[
\ig{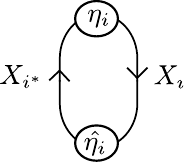} = 1
\]
for all $i \in I$. 
Define the complex numbers $a_i$ as the coefficients appearing in the associator expansion
\begin{equation}
\ig{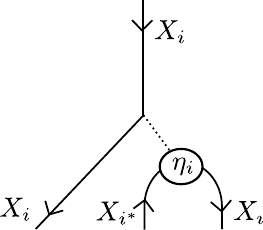} \, = a_i \! \ig{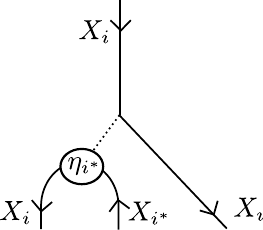} \quad + \text{ other terms}.
\end{equation}
Similarly define $b_i$ as the coefficients appearing in the inverse associator expansion
\begin{equation}
\ig{d3.pdf} \, = \,\, b_i \! \ig{d2.pdf} \quad + \text{ other terms}.
\end{equation}
Clearly, these coefficients are precisely the numbers appearing in
\begin{equation} \label{snake_numbers}
\ig{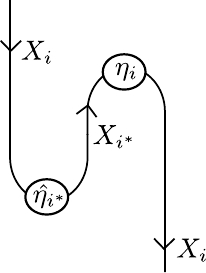} = a_i \, \ig{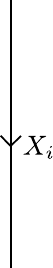} \, , \quad \quad \ig{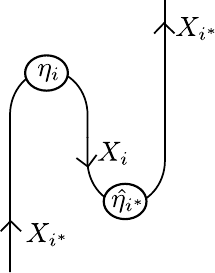} = b_i \, \ig{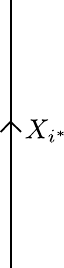} \, .
\end{equation}
In particular, they are nonzero as $\cat{C}$ is rigid (if they were zero, it would be impossible to choose counit maps satisfying the rigidity equations).
\begin{example} In the Yang-Lee category, we read off from \eqref{yl1} that $a_\tau = - \frac{1}{2}(1 + \sqrt{5})$. In category $\mathcal{E}$, we read off from \eqref{cate1} that $a_x = 1 + \sqrt{3}$; similarly $a_y = 1$.
\end{example}
\begin{lem} $b_{i} = a_{i^*}$ for all $i \in I$. \label{bequalsa}
\end{lem}
\begin{proof} Adapted from \cite[Prop 5.3.13]{bk01-ltc}.  Consider evaluating the morphism
\[
\ig{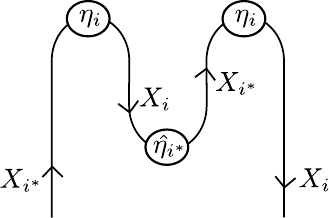}
\]
in two different ways. On the one hand, by dragging the right-most $\eta_i$ downwards, and then using (\ref{snake_numbers}a), it equals $a_i \eta_i$. On the other hand, by dragging the left-most $\eta_i$ downwards, and then using (\ref{snake_numbers}b), it equals $b_{i^*} \eta_i$. Hence $a_i = b_{i^*}$.
\end{proof}
For each $i \in I$, define $\epsilon_i : X_i \otimes X_{i^*} \rightarrow 1$ by
\[
\ig{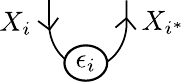} = \frac{1}{a_i} \ig{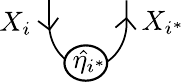} \, .
\]
We have just proved the following:
\begin{lem} For each $i \in I$, $(\eta_i, \epsilon_i)$ satisfy the rigidity equations furnishing $X_{i^*}$ as a right dual of $X_i$. 
\end{lem}
\begin{cor} The paired dimensions compute as $\displaystyle d_{\{i,i^*\}} = \frac{1}{a_{i^*}a_i}$. \label{paired_cor}
\end{cor}
\begin{lem} There exists a choice of basis of the $\eta_i$ such that $a_i = a_{i^*}$ for all $i \in I$.
\end{lem}
\begin{proof}
If $X_i$ is self-dual, then $i = i^*$ so we are done. Partition the non self-dual objects into ordered pairs $(X_i, X_{i^*})$. Scale the $\eta_i$ by setting $\eta'_i = \eta_i$ and $\eta'_{i^*} = \lambda_i \eta_{i^*}$ where $\lambda_i$ is a root of $\displaystyle \lambda_i^2 = \frac{a_i}{a_{i^*}}$. In this new basis we have $a'_i = a'_{i^*}$.
\end{proof}
In such a basis we have $a_i^{-2} = d_{\{i, i^*\}}$ for all $i \in I$. By Theorem \ref{realpositive}, the paired dimension $d_{\{i, i^*\}}$ is a positive real number, with the fusion dimension $d_i$ defined as its positive square root. Hence $\displaystyle a_i = \pm \frac{1}{d_i}$. If $X_i$ is not self-dual and $a_i = - d_i$, then we can remove this sign by setting $\eta'_i = -\eta_i$, $\eta'_{i^*} = \eta_{i^*}$, after which we have $a'_i = d_i$. We call such a basis choice $\eta_i \in \Hom(1, X_{i^*} \otimes X_i)$ satisfying $a_i > 0$ for all non self-dual $X_i$ a {\em fair basis}. 

Note that if $X_i$ is self-dual then this sign cannot be removed. We record this as a definition.

\begin{defn} Let $\cat{C}$ be a fusion category over $\mathbb{C}$. The {\em pivotal indicator} $p(X_i)$ of a simple object $X_i$ is defined as follows. If $X_i$ is not self-dual, $p(X_i) = 1$. If $X_i$ is self-dual, then $p(X_i)$ is the sign of the coefficient $a_i$ appearing in the associator expansion
\[
\ig{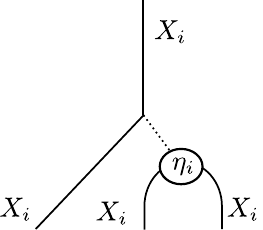} = a_i \! \ig{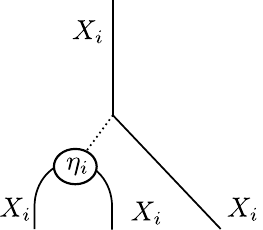} \quad + \text{ other terms}.
\]
where $\eta_i : 1 \rightarrow X_i \otimes X_i$ is some nonzero vector (the number $a_i$ is independent of this choice).
\end{defn}
\begin{example} In the Yang-Lee category, we read off from \eqref{yl1} that $p(\tau) = -1$. In category $\mathcal{E}$, we read off from \eqref{cate1} that $p(x)=1$, similarly $p(y)=1$.
\end{example}

\begin{rem} As in Remark \ref{remark_frobenius}, the pivotal indicator $p(X_i)$ of a self-dual simple object is closely related, but not equal to, the 2nd Frobenius-Schur indicator $\nu_2(X_i)$ of a self-dual simple object, as defined in \cite{ng2005higher}. The former is defined using only the underlying fusion category, while the latter depends on a pivotal structure $\gamma$, whence we can write it as $\nu_2^\gamma (X_i)$. The precise relationship is 
\[
 \nu_2^\gamma (X_i) = \gamma_i p(X_i)
\]
where $\gamma_i = \pm 1$ is the sign associated to the self-dual object $X_i$ in the pivotal structure $\gamma$ (see Section \ref{Explicit_equations}).
\end{rem}

\subsection{Formula for the pivotal operators} \label{Form_piv_op}
We now show how to compute the pivotal operators directly in terms of the associator matrix elements. More precisely, we express them in terms of the {\em apex associator monodromy}.
\begin{defn} The {\em apex associators} are the associator matrix elements where the top strand is the identity\footnote{Thus they form a sharp angle diagrammatically, hence the name.}:
\begin{equation} \label{apex_1}
 \ig{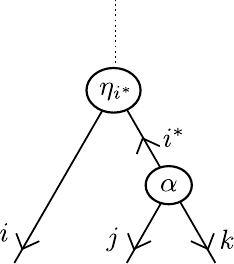} = \sum_\beta (S_{ijk})_{\beta \alpha} \ig{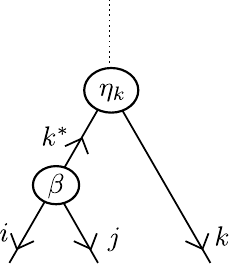}
\end{equation}
That is, $(S_{ijk})_{\beta \alpha} = (F^1_{ijk})_{\eta_{i^*} i^* \alpha}^{\eta_k k^* \beta}$. The {\em apex associator monodromy} $A_{ijk}$ is the product of matrices $S_{jki} \, S_{kij} \, S_{ijk}$.
\end{defn}
\begin{example} In the Yang-Lee category, we read off from \eqref{yl2} that $S_{\tau\tau\tau}=1$. For category $\mathcal{E}$, we read off from \eqref{cate2} and \eqref{cate3} that  
\[
S_{xxx} = \frac{1}{\sqrt{2}} e^{\frac{7 \pi i}{12}} \left( \begin{array}{cc} 1 & -i \\ -i & i \end{array} \right)
\]
and the apex associator monodromy is $A_{xxx} = S_{xxx}^3 = \id$.
\end{example}
\begin{rem} In all examples that the author knows of, the apex associator monodromy in a fair basis is the identity.
\end{rem}

\begin{prop} In a fair basis, the cyclic operators $C_{ijk}$ compute as
\[
C_{ijk} = p_k S_{ijk}
\]
where $p_k$ is the pivotal indicator of $X_k$ and $S_{ijk}$ are the apex associators.
\end{prop}
\begin{proof}
By definition, and using the pairing convention, the bending matrix elements $(C_{ijk})_{\beta \alpha}$ are the coefficients appearing in the expansion
\[
 d_k \, \ig{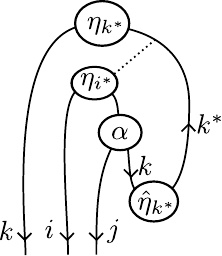} = \sum_\beta (C_{ijk})_{\beta \alpha} \ig{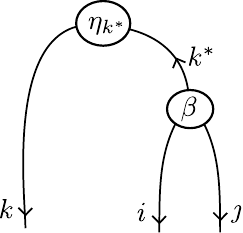} \, .
\]
We can rewrite the left-hand side using the apex associators \eqref{apex_1}:
\[
 \text{LHS} = d_k \sum_\beta (S_{ijk})_{\beta \alpha} \ig{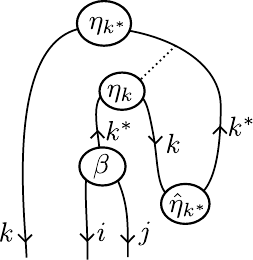} = d_k a_{k^*} \sum_\beta (S_{ijk})_{\beta \alpha} \ig{d16.pdf} 
\]
The second equality uses (\ref{snake_numbers}b) and Lemma \ref{bequalsa}. In a fair basis, $d_k a_{k^*} = p_k$ and we are done.
\end{proof}
We have thus proved the following result.
\begin{cor} In a fair basis, the monodromy of the cyclic operators is the product of the pivotal indicators with the apex associator monodromy:
\begin{equation}
 T_{ijk} = p_i p_j p_k A_{ijk} . 
\end{equation}
\end{cor}
Combining this with Lemma \ref{piv_op_lem_form} gives the following explicit formula for the pivotal operators $T^i_{jk}$ in terms of the apex associator monodromy.
\begin{thm} In a fair basis, the pivotal operator $T^i_{jk}$ is equal to the product of the pivotal indicators with the conjugate of the apex-associator monodromy,  \label{formula_piv}
\begin{equation} \label{form_for_piv_ops}
 T^i_{jk} =  p_i p_j p_k Y^{-1} A_{i^* j k} Y \,.
\end{equation}
\end{thm}

\begin{rem} The above formula \eqref{form_for_piv_ops} refines a formula of Wang \cite[Prop 4.17]{wang2010-topological}, which treats the multiplicity-free case and does not include the pivotal indicators. 
\end{rem}
\begin{example} In the Yang-Lee category, $T_{\tau \tau \tau} = (-1)^3 = -1$. In category $\mathcal{E}$, all the pivotal operators are the identity. 
\end{example}

\section{Pivotal structures} \label{piv_struc_sec}
In this section we characterize pivotal structures as solutions of an explicit set of equations over the complex numbers. We also define fusion homomorphisms, and discuss the sphericalization of a fusion category. 

\subsection{Explicit equations for pivotal structures} \label{Explicit_equations} 

In Section \ref{piv_monoidal} we showed that, by using the pairing convention, every fusion category over $\mathbb{C}$ comes equipped with a canonical monoidal endofunctor $T : \cat{C} \rightarrow \cat{C}$, the pivotal endofunctor, whose underlying functor is simply the identity, with the pivotal operators $T^i_{jk}$ supplying the coherence isomorphisms.

On the other hand, a {\em pivotal structure} on a fusion category is usually defined (see eg. \cite{eno02-ofc, bw99-sc}) by making choices of right duals $(V^*, \eta_V, \epsilon_V)$ for every object $V$ as in Section \ref{pivops}, giving rise to a monoidal double dual functor $** : \cat{C} \rightarrow \cat{C}$, and then declaring that a pivotal structure is a monoidal natural isomorphism $\gamma : \id \Rightarrow **$.  In Section \ref{pivops} we showed that there is a canonical monoidal natural isomorphism $\mathcal{T} \cong **$. Thus, to avoid fixing choices of duals, in this paper we adopt the following cleaner definition. 

\begin{defn} A pivotal structure on a fusion category over $\mathbb{C}$ is a monoidal natural isomorphism $\gamma: \id \Rightarrow \mathcal{T}$, where $\mathcal{T}$ is the canonical pivotal endofunctor of the category.
\end{defn}
For an alternative formulation in terms of {\em even-handed structures}, see \cite{BBthesis}. Under this definition, quantum dimensions are defined as follows. 
\begin{defn} Let $\gamma$ be a pivotal structure on a fusion category. The {\em quantum dimension} of an object $V$ with respect to $\gamma$ is defined as
\[
 \dim_\gamma (V) = \ig{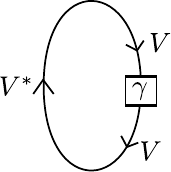}
\]
where the pairing convention has been used. The pivotal structure $\gamma$ is {\em spherical} when $\dim_\gamma (V) = \dim_\gamma (V^*)$ for all objects $V$.
\end{defn}
The following is clear.
\begin{lem} The quantum dimension of a simple object $X_i$ with respect to a pivotal structure $\gamma$ computes as $\dim_\gamma(X_i) = \gamma_i d_i$ where $d_i$ is the fusion dimension of $X_i$. \label{dimension_simple}
\end{lem}

Under these conventions, we have the following.

\begin{thm} A pivotal structure on fusion category $\cat{C}$ with representative simple objects $X_i$, $i \in I$ corresponds to a collection of numbers $\gamma_i \in U(1)$, $i \in I$ satisfying
\begin{equation} \label{piveqn2}
 \gamma_j \gamma_k \id = \gamma_i T^i_{jk} \quad \text{whenever $X_i$ is a summand of $X_j \otimes X_k$} 
\end{equation}
where the $T^i_{jk}$ are the pivotal operators of the fusion category. Moreover, the pivotal structure is spherical precisely when $\gamma_i = \pm 1$ for all $i \in I$. \label{piv_structure_thm}
\end{thm}
\begin{proof}
By Lemma \ref{extend}, $\gamma$ is uniquely determined by its components $\gamma_i = \gamma_{X_i}$ on the representative simple objects $X_i$. The coherence equation for being a monoidal natural isomorphism reduces to the scalars $\gamma_i \in \mathbb{C}^\times$ obeying \eqref{piveqn2}. Given a solution to \eqref{piveqn2}, the collection $\{ \pm \gamma_i\}$ forms a finite subgroup of $\mathbb{C}^\times$ and hence $\gamma_i \in U(1)$.  Since $T^1_{11} = \id$ from Lemma \ref{triv_pivs}, we have $\gamma_1 = 1$, and similarly since $T^1_{ii^*} = \id$, we have $\gamma_i = \overline{\gamma_{i^*}}$ for all $i \in I$. From Lemma \ref{dimension_simple}, if $\gamma$ is spherical, then $\gamma_i d_{i^*} = \gamma_{i^*} d_{i^*}$ whence $\gamma_i = \gamma_{i^*}$ since $d_{i^*} = d_i$. Hence $\gamma_i = \pm 1$.  
\end{proof}
\begin{cor}[{cf. \cite[Prop 8.23]{eno02-ofc}}] A pseudo-unitary fusion category admits a canonical spherical structure. \label{pseudo_cor}
\end{cor}
\begin{proof}
Follows from Corollary \ref{pseudo_unitary}.
\end{proof}

\begin{defn} We say that a fusion category $\cat{C}$ over $\mathbb{C}$ is {\em orientable} if the pivotal symbols $\epsilon^i_{jk, \alpha}$ do not depend on $\alpha$, that is, $T^i_{jk} = \epsilon^i_{jk} \id$ for some signs $\epsilon^i_{jk} = \pm 1$.
\end{defn}
Clearly we have the following.
\begin{lem} If a fusion category is not orientable, then it does not admit a pivotal structure. 
\end{lem}
Recall from Theorem \ref{formula_piv} that the pivotal operators can be expressed in terms of the apex associator monodromy operators.
\begin{prop} If the apex associator monodromy $A_{ijk} = \id$ for all $i,j,k \in I$, then the fusion category admits a canonical spherical structure.
\end{prop}
\begin{proof} If $A_{ijk} = \id$, then from Theorem \ref{formula_piv}, 
\[
T^i_{jk} = p_i p_j p_k \id
\]
where the $p_i = \pm 1$ are the pivotal indicators. Hence setting $\gamma_i = p_i$ will solve \eqref{piveqn2}.
\end{proof}

\subsection{Fusion homomorphisms}
We write $[\cat{C}]$ for the Grothendieck ring of a fusion category. The following property of quantum dimensions is well-known.
\begin{lem} Given a pivotal structure $\gamma$ on a fusion category $\cat{C}$, the quantum dimension map $\dim_{\gamma} \colon [\cat{C}] \rightarrow
\mathbb{C}$ which sends $[V] \mapsto \dim_\gamma (V)$ satisfies the following:
 \begin{itemize}
  \item It is a ring homomorphism,
  \item $\dim_\gamma [X_i] \dim_\gamma [X_i^*] = d_{\{i,i^*\}}$ for all
simple objects $X_i$.
 \end{itemize}
\end{lem}
This motivates the following definition.
\begin{defn} Let $\cat{C}$ be a fusion category. We call a function $f \colon [\cat{C}] \rightarrow \mathbb{C}$ a {\em fusion homomorphism} if it is a ring homomorphism and if $f[X_i]
f[X_i^*] = d_{\{i,i^*\}}$ for all simple objects $X_i$. 
\end{defn}
For a fusion category $\cat{C}$, it is interesting to consider whether the injective map
\[
 \text{Pivotal structures}(\cat{C}) \rightarrow \{\text{Fusion homomorphisms $f : [\cat{C}] \rightarrow \mathbb{C}$}\}
\]
is also surjective. The following case is instructive. Fix a finite group $G$.
\begin{prop} The pivotal structures on $\Rep(G)$ are in 1-1 correspondence
with the fusion homomorphisms from $[\Rep(G)]$ to $\mathbb{C}$.
\end{prop}
\begin{proof} We have
 \begin{align*}
  \text{Pivotal structures}(\Rep(G)) & \cong \Aut_\otimes(\id) \\
   & \cong Z(G) \\
   & = \{ g \in G \colon |\Tr_{V_i} (g) | = \dim V_i \text{ for all
irreducibles $V_i$}\} \\
   & \cong \{\text{Fusion homomorphisms } f \colon [\Rep(G)] \rightarrow
\mathbb{C} \}.
  \end{align*}
The first isomorphism uses the fact the fact that $\Rep(G)$ comes with a canonical pivotal structure, and that in general the set of pivotal structures is a torsor for $\Aut_\otimes (\id_C)$. The second isomorphism is a result
of M\"{u}ger \cite{muger2004center}. The equality in the third
line is a basic result of representation theory \cite[Cor 2.28]{isaacs1976-character}. The final isomorphism uses two facts. Firstly,
every ring homomorphism
 \[
  [\Rep(G)] \rightarrow \mathbb{C}
 \]
must take the form $V \mapsto \Tr_V(g)$ for some fixed $g \in G$ ---
because a character like this certainly {\em is} a ring homomorphism,
and there are as many such distinct characters as there are conjugacy
classes in the group, which must exhaust all the ring homomorphisms
since $[\Rep(G)]_\mathbb{C}$ is isomorphic to the space of functions
on the conjugacy classes. Secondly, in $\Rep(G)$ the paired dimensions
$d_{\{i,i^*\}}$ are just $\dim(V_i)^2$, so that a fusion homomorphism
must satisfy $| f(V_i) | = \dim V_i$.
\end{proof}

\subsection{Sphericalization of a fusion category}
We have seen in Corollary \ref{piv_operator_C}  that every fusion category $\cat{C}$ over $\mathbb{C}$ comes equipped with a canonical monoidal action of $\mathbb{Z} / 2 \mathbb{Z}$. The `equivariantization' with respect to this action defines a new fusion category admitting a canonical spherical structure. 

Recall (see eg. \cite{nikshych2008non}) that a monoidal action of a group $G$ on a monoidal category $\cat{C}$ is a monoidal functor $F : BG \rightarrow  \Aut_\otimes(\cat{C})$, where $BG$ is the group $G$\ thought of as a one-object category, and $\Aut_\otimes (\cat{C})$ is the monoidal category of monoidal endofunctors of $\cat{C}$. 

\begin{defn} The {\em equivariantization} $\cat{C}^G$ is the monoidal  category defined as follows:
\begin{itemize}
\item An object of $\cat{C}^G$ consists of an object $X \in \cat{C}$ together with isomorphisms $s_g : F_g(X) \rightarrow X$ for each $g \in G$ such that the diagram
\begin{equation} \label{Gequivob}
\ba
\begin{tikzpicture}[xscale=3, yscale=1.4]
\node (1) at (0,0) {$F_g F_h(X)$};
\node (2) at (1,0) {$F_g(X)$};
\node (3) at (0,-1) {$F_{gh} (X)$};
\node (4) at (1,-1) {$X$};
\draw[->] (1) -- node[above] {$F_g(s_h)$} (2);
\draw[->] (2) -- node[right] {$s_g$} (4);
\draw[->] (3) -- node[below] {$s_{gh}$} (4);
\draw[->] (1) -- node[left] {$\gamma(g,h)_X$} (3);
\end{tikzpicture}
\ea
\end{equation}
commutes for each $g,h \in G$.
\item A morphism in $\cat{C}^G$ is a morphism $f : X \rightarrow Y$ in $C$ satisfying 
\begin{equation} \label{morphismCG}
f \circ s^X_g = s^y_g \circ F_g(f)
\end{equation}
 for each $g \in G$. 
\item The tensor product is defined by $(X, s^X) \otimes (Y, s^Y) = (X \otimes Y, s^{X\otimes Y})$ where $s^{X \otimes Y}_g$ is the composite
\begin{equation}
\ba \begin{tikzpicture}[xscale=3.6]
 \node (1) at (0,0) {$F_g(X \otimes Y)$};
 \node (2) at (1,0) {$F_g(X) \otimes F_g(Y)$};
 \node (3) at (2,0) {$X \otimes Y$}; 
 \draw (1) -- node[above]{$(T^g)^{-1}_{X,Y}$} (2);
 \draw (2) -- node[above] {$s^X_g \otimes s^Y_g$} (3);
  \end{tikzpicture} \ea \, 
\end{equation}
where $T^g_{X,Y} : F_g(X) \otimes F_g(Y) \rightarrow F_g(X \otimes Y)$ are the coherence isomorphisms equipping $F_g$ as a monoidal functor.
\end{itemize}
\end{defn}

\begin{defn} The {\em sphericalization} $\tilde{\cat{C}}$ of a fusion category
$\cat{C}$ is its equivariantization with respect to the canonical $\mathbb{Z}
/ 2 \mathbb{Z}$ action on it.
\end{defn}
 
The advantage of our approach is that the action of $\mathbb{Z} / 2 \mathbb{Z}$ is especially simple, as the group acts by identity functors, with the monoidal coherence isomorphisms encoded in the pivotal operators $T^i_{jk}$.

\begin{lem} The sphericalization $\tilde{\cat{C}}$ of $\cat{C}$ has simple objects $X_i^{s_i}$ where $X_i$ is a simple object in $C$, and $s_i = \pm 1$. The fusion hom-vector spaces in $\tilde{\cat{C}}$ compute in terms of the decompositions of the fusion hom-sets in $\cat{C}$ (see \eqref{decomp_eqn}) as
\begin{equation} \label{fusion_hom_sets}
 \Hom_{\tilde{C}} (X_i^{s_i}, X_j^{s_j} \otimes X_k^{s_k}) = \Hom_C (X_i, X_j \otimes X_k)_{s_i s_j s_k} \, .
\end{equation}
The associator in $\tilde{\cat{C}}$ is the pullback of the associator in $\cat{C}$. The  forgetful tensor functor $F : \tilde{\cat{C}} \rightarrow \cat{C}$ sends $X_i^\sigma \mapsto X_i$. The pivotal symbols $\tilde{\epsilon}$ of $\tilde{\cat{C}}$ compute in terms of the pivotal symbols $\epsilon$ of $\cat{C}$ as 
\begin{equation} \label{epsilon_tilde}
 \tilde{\epsilon}^{(i, s_i)}_{(j, s_j) (k, s_k)} = s_i s_j s_k \, .
\end{equation}
\end{lem}
\begin{proof} Write $\mathbb{Z} / 2 \mathbb{Z} = \{1, -1\}$. The action of $\mathbb{Z}/2\mathbb{Z}$ on $C$ is strict as a group action, so $\gamma=\id$ in \eqref{Gequivob}. For a simple object $X_i$, write $s_i = s^{X_i}_{-1}$. Then \eqref{Gequivob} becomes $s_i^2 = 1$ so $s_i = \pm 1$. Let $e_\alpha : X_i \rightarrow X_j \otimes X_k$ be a pivotal basis. From \eqref{morphismCG}, $e_\alpha \in \Hom(X_i^{s_i}, \, X_j^{s_j} \otimes X_k^{s_k})$ if and only if 
\begin{equation} \label{fusion_2}
\begin{tz}[xscale=3, yscale=1.3]
 \node (1) at (0,0) {$X_i$};
 \node (2) at (1,0) {$X_j \otimes X_k$};
 \node (3) at (2,0) {$X_j \otimes X_k$};
 \node (4) at (2,-1) {$X_j \otimes X_k$};
 \node (5) at (0,-1) {$X_i$};
 \draw[->] (1) -- node[above]{$e_\alpha$} (2);
 \draw[->] (2) -- node[above]{$T_{X_j\otimes X_k}$} (3);
 \draw[->] (3) -- node[right]{$s_j s_k$} (4);
 \draw[->] (1) -- node[left]{$s_i$} (5);
 \draw[->] (5) -- node[below]{$e_\alpha$} (4);
\end{tz}
\end{equation}
commutes. By definition (see Corollary \ref{extend_cor}), $T_{X_j X_k} \circ e_\alpha = \epsilon^i_{jk, \alpha} e_\alpha$, so that \eqref{fusion_2} gives \eqref{fusion_hom_sets}, as well as \eqref{epsilon_tilde}.
\end{proof}
\begin{cor}[{cf. \cite[Prop 5.14]{eno02-ofc}}] The sphericalization $\tilde{\cat{C}}$ of a fusion category $\cat{C}$ carries a canonical spherical structure, in which the quantum dimension of a simple object $X^{s_i}_i$ in $\tilde{\cat{C}}$ computes in terms of the fusion dimension $d_i$ of $X_i$ in $\cat{C}$ as \label{spher_cor}
\begin{equation} \label{qdimtilde}
\dim (X^{s_i}_i) = s_i d_i \, .
\end{equation}
\end{cor}
\begin{proof}
By \eqref{epsilon_tilde}, $\tilde{\cat{C}}$ is orientable. So from Theorem \ref{piv_structure_thm}, a spherical structure amounts to a choice of signs $t_{(i, s_i)} = \pm 1$ satisfying 
\[
t_{(j, s_j)} t_{(k, s_k)} = s_i s_j s_k t_{(i, s_i)}
\]
whenever $\Hom(X_i, X_j \otimes X_k)_{s_i s_j s_k}$ is nonzero. Clearly the choice $t_{(i, s_i)} = s_i$ provides a solution, and \eqref{qdimtilde} follows from the definition of the quantum dimension.
\end{proof}
\begin{example} The sphericalization of the Yang-Lee category has simple objects $1^+$, $1^-$, $\tau^+$, $\tau^-$. The fusion rules are $\tau^+ \tau^+ = 1^-$, $1^- 1^- = 1^+$, $\tau^{\pm} 1^- = 1^- \tau^{\pm} = \tau^{\mp}$, and $\dim (\tau^{\pm}) = \pm d_\tau = \pm (\phi - 1)$.
\end{example}
\begin{lem} The Frobenius-Perron and fusion dimensions in $\tilde{\cat{C}}$ are the same as in $\cat{C}$, that is, $d^+_{(X_i, s_i)} = d^+_{X_i}$ and $d_{(X_i, s_i)} = d_i$. \label{sphericalization_dim}
\end{lem}
\begin{proof} We claim that setting $d^+_{(X_i, s_i)} = d^+_{X_i}$ solves \eqref{FPeqn} and hence is the unique solution, by the Frobenius-Perron theorem. Indeed,
\begin{align*}
 d^+_{(X_j, s_j)} d^+_{(X_k, s_k)} &= d^+_j d^+_k \\ 
  &= \sum_i N^i_{jk} d^+_i \\
  &= \sum_i \left[ (N^i_{jk})_{s_j s_k +} + (N^i_{jk})_{s_j s_k -}  \right] d^+_i \quad \text{for all $s_j,s_k = \pm 1$} \\
  &= \sum_{(i, s_i)} (N^i_{jk})^{s_j s_k s_i} d^+_{(X_i, s_i)} \, .
\end{align*}
For the fusion dimensions, the associator in $\tilde{\cat{C}}$ is the pullback of that in $\tilde{\cat{C}}$, so the paired dimensions are same as in $\cat{C}$, i.e. $d_{\{(X_i, s_i), (X_i^*, s_i)\}} = d_{\{i, i^*\}}$. 
\end{proof}

\section{Ocneanu rigidity} \label{ocneanu_sec}
In \cite[Appendix E6]{k05-aes} Kitaev gave a diagrammatic proof that the Davydov-Yetter cohomology \cite{yetter1998-braided, yetter2003-abelian, davydov1997-twisting} of a unitary fusion category vanishes in positive degrees. In this section we show that the pairing convention enables us to extend Kitaev's graphical proof to all fusion categories. 

\subsection{The tangent complex}
In this subsection we recall the definition of the Yetter-Davydov cohomology of a fusion category $\cat{C}$ \cite{yetter1998-braided, yetter2003-abelian, davydov1997-twisting}. Define the functor 
\[
T_n : \underbrace{\cat{C} \boxtimes \cdots \boxtimes \cat{C}}_{n \text{ times}} \rightarrow \cat{C}
\]
by $T_n(A_1, \ldots, A_n) = A_1 \otimes \cdots \otimes A_n$, parenthesized from left-to-right for definiteness, and define $C^n = \End(T_n)$, the vector space of natural endomorphisms of $T_n$. For $n=0$, set $C^0 = \mathbb{C}$. In terms of representative simple objects $X_i$, $i \in I$, we have
\[
C^n = \bigoplus_{i_1, \ldots, i_n} \Hom(X_1 \otimes \cdots \otimes X_n, \, X_1 \otimes \cdots \otimes X_n) \, .
\]
In string diagrams, we can write the components of $c \in C^n$ as
\[
\ig{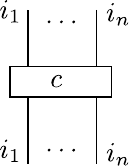} \, .
\]
For instance, a 1-cochain $c \in C^1$ is simply a collection of scalars $c_i \in \mathbb{C}$, $i \in I$. A 2-cochain $A \in C^2$ is a collection of morphisms $A_{ij} : X_i \otimes X_j \rightarrow X_i \otimes X_j$; this amounts to a collection of linear operators
\begin{equation} \label{2cocycle1}
A^i_{jk} : \Hom(X_k, X_i \otimes X_j) \rightarrow \Hom(X_k, X_i \otimes X_j) \, .
\end{equation}

\begin{defn} The {\em tangent complex} of a fusion category $\cat{C}$ is the sequence of $\mathbb{C}$-vector spaces and linear maps
\[
C^0 \stackrel{d^0}{\longrightarrow} C^1 \stackrel{d^1}{\longrightarrow} C^2 \stackrel{d^2}{\longrightarrow} C^3 \rightarrow \ldots, \qquad d^n = \sum_{k=0}^{n+1}(-1)^kf^n_k
\]
where the maps $f^n_k : C^n \rightarrow C^{n+1}$ are defined by
\[
\ig{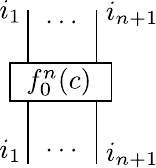} = \ig{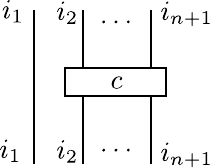}, \qquad \ig{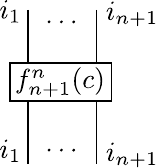} = \ig{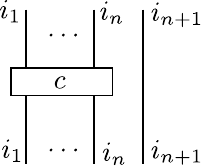} 
\]
\[
\ig{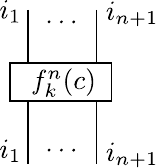} = \sum_{k, \alpha} \ig{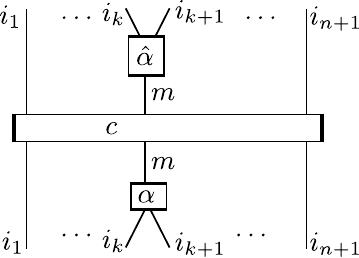}
\]
(Note that for $c \in C^0 = \mathbb{C}$, $(f^0_0(c))_i = \id_{X_i} = (f^0_1(c))_i$ so that $d^0 = 0$.)
\end{defn}
\begin{lem} The tangent complex is a complex, i.e. $d^{n+1} d^n = 0$. 
\end{lem}
\begin{proof}
Follows from the identity 
\begin{equation} \label{chain_eqn}
f^{n+1}_k f^n_m = f^{n+1}_{m+1} f^n_k \quad 0 \leq k \leq m \leq n+1 \, . 
\end{equation}
For $k=m=0$ or $k=m=n+1$, \eqref{chain_eqn} is just the resolution of the identity \eqref{res_id}. For $k \neq m$, \eqref{chain_eqn} is just the interchange law in a monoidal category. In the other cases, \eqref{chain_eqn} takes the following form (we do the case $k=m=1$, $n=2$):
\begin{equation} \label{chain2}
\sum_{m, \alpha, p, \beta} \ig{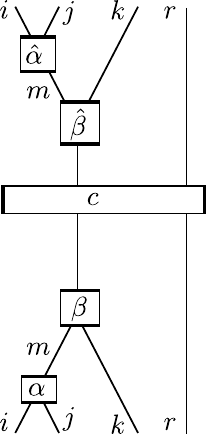} = \sum_{n, \delta, q, \gamma} \ig{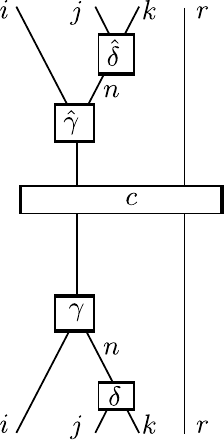}
\end{equation}
Pre-composing both sides of \eqref{chain2} with
\[
\ig{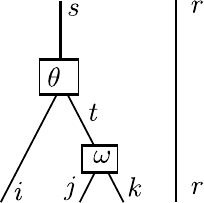}
\]
and inserting the associator expansion \eqref{associator_eqn}, the left and right sides become the same expression, namely
\[
\sum_{m, \alpha, \beta} (F^s_{ijk})^{\beta m \alpha}_{\theta t \omega} \ig{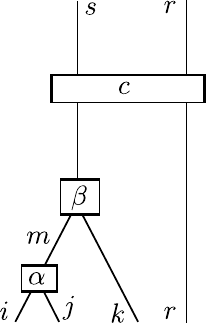} \, .
\]
\end{proof}
\begin{example} \label{2-cocycle-example}A 1-cocycle $c \in Z^1(C)$ is a collection of numbers $c_i$ such that $c_k = c_i + c_j$ whenever $X_k$ is a summand of $X_i \otimes X_j$. A 2-cocycle $A \in Z^2(C)$ is a collection of linear operators $A^i_{jk}$ as in \eqref{2cocycle1} such that 
\[
\ig{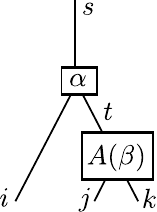} - \ig{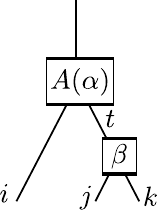} + \sum_{\gamma, m, \delta} (F^s_{ijk})^{\gamma m \delta}_{\alpha t \beta} \left[ \ig{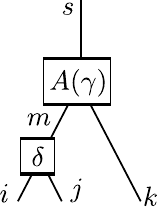} - \ig{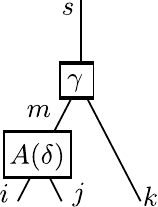} \right] = 0.
\]
Expanded out fully, this is the equation
\begin{equation} \label{2cocycle2}
\sum_\alpha (A^t_{jk})_{\alpha \omega} F^{\gamma' m' \delta'}_{\theta t \alpha} - \sum_\beta (A^s_{it})_{\beta \theta} F^{\gamma' m' \delta'}_{\beta t \omega} + \sum_\gamma (A^s_{m' k})_{\gamma' \gamma} F^{\gamma m' \delta'}_{\theta t \omega} - \sum_\delta (A^m_{ij})_{\delta' \delta} F^{\gamma' m' \delta}_{\theta t \omega}
\end{equation}
for all values of the free indices, where $F\equiv F^s_{ijk}$.
\end{example}

\begin{defn} The {\em Yetter-Davydov} cohomology $H^i(\cat{C})$ of a fusion category $\cat{C}$ is the cohomology of its tangent complex.
\end{defn}
The low-dimensional cohomology groups have the following interpretations \cite{yetter1998-braided, yetter2003-abelian, davydov1997-twisting, k05-aes}. The group $H^1(\cat{C})$ classifies first-order deformations of $\Aut_\otimes (\id_{\cat{C}})$, $H^2(\cat{C})$ classifies first-order deformations of the tensor product functor, and $H^3(\cat{C})$ classifies first-order deformations of the associator.

\subsection{Vanishing of cohomology}
We will need the following `handleslide' identity.
\begin{lem} In the pairing convention for the graphical calculus associated to a fusion category, we have
\[
\sum_{\alpha, p}  d_p \, \ig{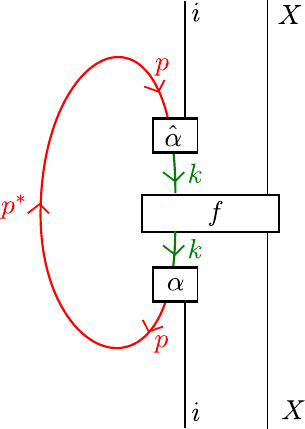} = \sum_k d_k \ig{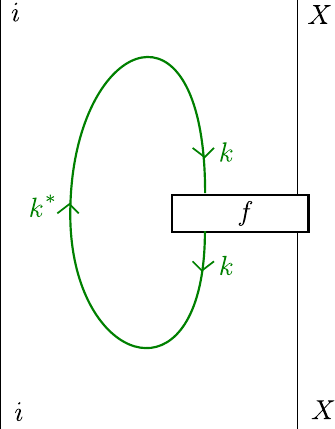}
\]
where $d_p$ are the fusion dimensions, $X$ is an arbitrary object, and $f: X_k \otimes X \rightarrow X_k \otimes X$ is an arbitrary morphism. \label{handleslide} 
\end{lem}
\begin{proof}
\begin{equation} \label{snakey_2}
\text{LHS } = \sum_{p, \alpha} d_p \ig{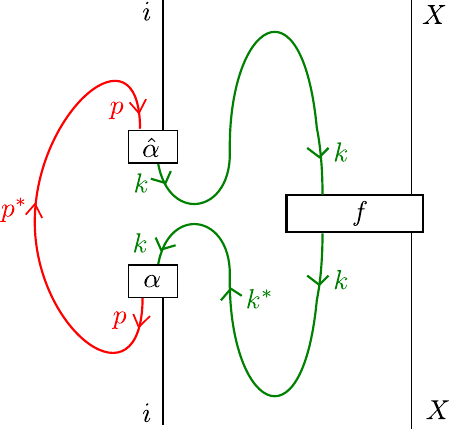} = d_k \sum_{p, \alpha} \ig{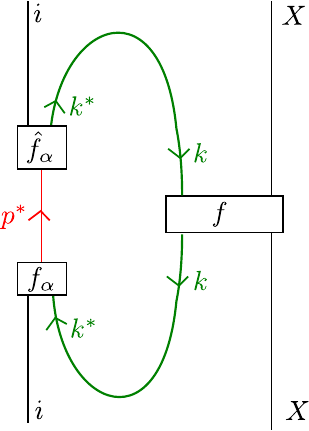}
\end{equation}
where in the first equality the pairing convention is being used, and in the second equality we transformed the old basis vectors $e_\alpha$ and $\hat{e}_\alpha$ (which were written just as $\alpha$ and $\hat{\alpha}$ in the diagrams) into new basis vectors $f_\alpha \in \Hom(X_{p^*}, X_i \otimes X_{k^*})$ and $\hat{f}_\beta \in \Hom(X_i \otimes X_{k^*}, X_{p^*})$ by the formulas
\begin{equation} \label{snakey_dual}
\ig{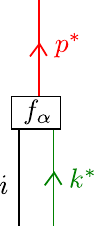} := \ig{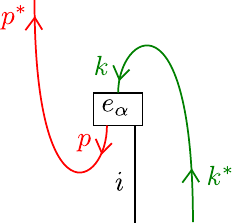} \qquad \ig{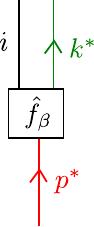} := \frac{d_p}{d_k} \ig{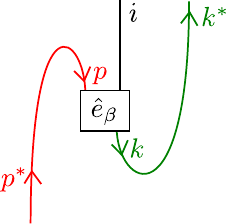}
\end{equation}
where the cup and cap maps are the same ones used in \eqref{snakey_2}. Now we claim that $f_\alpha$ and $\hat{f}_\beta$ form a dual basis. Indeed,
\[
\ig{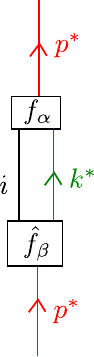} \,\,= \frac{d_p}{d_k} \ig{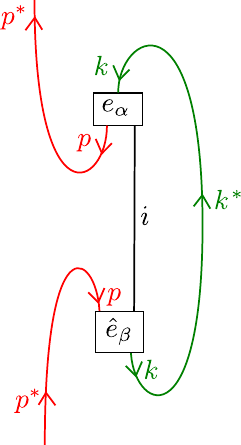} = \delta_{\alpha \beta} \ig{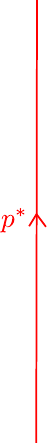}
\]
as can be seen by closing the loop on the left and applying the definition of the pairing convention. Hence the $f_\alpha \hat{f}_\alpha$ term in the right hand side of \eqref{snakey_2} provides a resolution of the identity, proving the lemma.
\end{proof}

\begin{thm}[{cf. \cite[Thm 2.27]{eno02-ofc}}] $H^n(\cat{C}) = 0$ for all $n > 0$. \label{ocneanu}
\end{thm}
\begin{proof} We define an operator $\chi^n : C^n \rightarrow C^{n-1}$ by 
\[
\ig{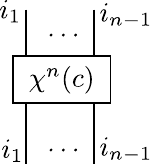} = \frac{1}{D} \sum_p d_p \ig{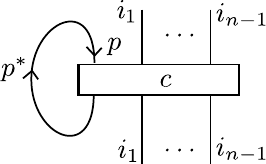}
\]
where on the right hand side we are using the pairing convention (else this diagram would not be well-defined), $d_p$ is the fusion dimension of $X_p$, and $D$ is the global dimension of $\cat{C}$. We will show that
\begin{equation} \label{chi_eqn}
\chi^{n+1} f^n_k =  \begin{cases} \id_{C^n} & \text{ if } k=0 \\ f^{n-1}_{k-1} \chi^n & \text{ if } k,n > 0 \end{cases}
\end{equation}
which will imply that $d\chi + \chi d = \id$, and hence the cohomology vanishes. The case $k=0$ of \eqref{chi_eqn} is simply the relation $D = \sum_p d_p^2$, which is the definition of $D$. The cases $k > 1$ are tautologies. The nontrivial case is $k=1$, which is precisely the handleslide identity from Lemma \ref{handleslide}.
\end{proof}
\begin{example} For a 2-cocycle $A \in Z^2(C)$, that is, a collection of linear operators $A^i_{jk}$ satisfying \eqref{2cocycle2},the vanishing of $H^2(\cat{C})$ means that $A^i_{jk}$ must take the form
\begin{equation} \label{forced2}
 A^i_{jk} = (c_j - c_i + c_k) \id
\end{equation}
for some scalars $c_i, c_j, c_k$. This conclusion is perhaps not immediately evident from \eqref{2cocycle2}. Even if we assume that the $A^i_{jk}$ in \eqref{2cocycle2} are diagonal (which is also not immediately clear from $d A = 0$) i.e. $(A^i_{jk})_{\alpha \beta} = \delta_{\alpha \beta} \, a^i_{jk, \alpha}$, then \eqref{2cocycle2} becomes the equation
\begin{equation} \label{forcing_associators}
a^t_{jk, \omega} - a^s_{it, \theta} + a^s_{mk, \gamma} - a^m_{ij, \delta} = 0 \quad \text{whenever } (F^s_{ijk})^{\theta t \omega}_{\delta m \gamma} \neq 0.
\end{equation}
The vanishing of $H^2(\cat{C})$ means that the equations \eqref{forcing_associators} must force the $a^i_{jk}, \alpha$ to all be equal for different $\alpha$ and furthermore take the form \eqref{forced2}. Thus Ocneanu rigidity implies in particular that sufficiently many associators must be nonzero.
\end{example}

\bibliographystyle{plainurl}
\bibliography{references}

\Addresses

\end{document}